\documentclass{amsart}

\usepackage{amssymb}
\usepackage{amsfonts}
\usepackage{amsmath}
\usepackage{amsthm}
\usepackage{graphicx}
\usepackage{mathrsfs}
\usepackage{dsfont}
\usepackage{amscd}
\usepackage{multirow}
\usepackage{palatino}
\usepackage{mathpazo}
\usepackage{mathtools}
\usepackage[all]{xy}
\usepackage[scaled]{helvet} 
\usepackage{courier} 
\normalfont
\usepackage[T1]{fontenc}
\usepackage{calligra}
\usepackage{verbatim}
\usepackage[usenames,dvipsnames]{color}
\usepackage[colorlinks=true,linkcolor=Blue,citecolor=Violet]{hyperref}
\usepackage{cite}
\usepackage{enumitem}

\makeatletter
\@namedef{subjclassname@2010}{%
  \textup{2010} Mathematics Subject Classification}
\makeatother

\theoremstyle{plain}
\newtheorem{theorem}{Theorem}[section]

\newtheorem{corollary}[theorem]{Corollary}

\newtheorem{lemma}[theorem]{Lemma}
\newtheorem{proposition}[theorem]{Proposition}

\newtheorem{theorem-definition}[theorem]{Theorem-Definition}

\theoremstyle{definition}
\newtheorem{definition}[theorem]{Definition}

\newtheorem{convention}[theorem]{Convention}

\theoremstyle{remark}

\newtheorem{example}[theorem]{Example}

\newtheorem{remark}[theorem]{Remark}

\numberwithin{equation}{section}
 
\newcommand{\A}{{\mathcal{A}}}
\newcommand{\wt}[1]{{\widetilde{#1}}}
\newcommand{\un}[1]{{1_{\widetilde{#1}}}}

\newcommand{\ee}{{\varepsilon}}
\newcommand{\la}{{\lambda}}
\newcommand{\La}{{\Lambda}}
\newcommand{\al}{{\alpha}}
\newcommand{\be}{{\beta}}
\newcommand{\ga}{{\gamma}}

\newcommand{\om}{{\omega}}

\newcommand{\de}{{\delta}}

\newcommand{\vhi}{{\varphi}}

\newcommand{\dd}{{\mathsf{d}}}

\newcommand{\N}{{\mathds{N}}}

\newcommand{\R}{{\mathds{R}}}
\newcommand{\C}{{\mathds{C}}}

\newcommand{\M}{{\mathfrak{M}}}

\newcommand{\dom}[1]{{\operatorname*{dom}({#1})}}

\renewcommand{\geq}{\geqslant}
\renewcommand{\leq}{\leqslant}

\allowdisplaybreaks[1]

\makeatletter
\newcommand{\vast}{\bBigg@{4}}
\newcommand{\Vast}{\bBigg@{5}}
\makeatother

\begin{document}

\title[Bures metric, quantum metric, non-unital C*-algebras]{The Bures metric and the quantum metric on the density space of a C*-algebra: the non-unital case}
\author[K. Aguilar]{Konrad Aguilar}
\address{Department of Mathematics and Statistics, Pomona College, 610 N. College Ave., Claremont, CA 91711} 
\email{konrad.aguilar@pomona.edu}
\urladdr{\url{https://aguilar.sites.pomona.edu}}
\thanks{The  first  author is partially supported by NSF grant DMS-2316892. The fourth author is partially supported by the Research Council of Norway (RCN), project no.~345433.}
 
\author[K. Behera]{Karina Behera}
\address{Department of Mathematics and Statistics, Pomona College, 610 N. College Ave., Claremont, CA 91711} 
\email{kaba2021@mymail.pomona.edu}

\author[K. von Bornemann Hjelmborg]{Katrine von Bornemann Hjelmborg}
\address{Department of Mathematics and Computer Science, The University of Southern Denmark, Campusvej 55, DK-5230 Odense M, Denmark}
\email{kahje@imada.sdu.dk}

\author[T. Omland]{Tron Omland}
\address{Norwegian National Security Authority (NSM) and Department of Mathematics, University of Oslo, Norway}
 \email{tron.omland@gmail.com}

 \author[G. Wickham]{Gregory Wickham}
 \address{Department of Mathematics, Harvey Mudd College, 320 E. Foothill Blvd., Claremont, CA 91711}
 \email{gwickham@g.hmc.edu}

  \author[N. Wu]{Nicole Wu}
 \address{Department of Mathematics, Harvey Mudd College, 320 E. Foothill Blvd., Claremont, CA 91711}
 \email{nwu@hmc.edu}

 \author[A. M. Yassine]{Adam M. Yassine}
 \address{Department of Mathematics and Statistics, Pomona College, 610 N. College Ave., Claremont, CA 91711}
 \email{Adam.Yassine@pomona.edu}


\maketitle

\begin{abstract}
Building off work of Farenick and Rahaman, we extend the definition of the density space and the Bures metric to the setting of non-unital C*-algebras equipped with a faithful trace and prove that the Bures metric is also a metric in this case and show that its topology is weaker than the topology induced by the C*-norm. Furthermore, we prove a Heine-Borel type theorem for C*-algebras   and the density space. In particular, we prove that for any C*-algebra (unital or non-unital) equipped with a faithful trace, the density space equipped with the Bures metric topology is not compact if and only if the C*-algebra is infinite dimensional. We also exhibit several examples of sequences that have no converging sequence in the unital and non-unital case including both commutative and noncommutative C*-algebras. Next, building off work from some of the authors, we extend the definition of the quantum metric on the density space to the non-unital C*-algebra case by introducing the notion of a quantum Lipschitz triple, which form a subclass of quantum locally compact metric spaces of Latr\'emoli\`ere that   utilize Rieffel's notion of a quantum metric (we also introduce new classes of quantum locally compact metric spaces that include certain noncommutative homogeneous C*-algebras). Furthermore, we prove that this quantum metric topology is weaker than the topology of the one induced by the C*-norm and finish the article with an analysis of matrix-valued functions on the quantized interval, which provides commutative and noncommutative examples where the quantum metric topology on the density space is not compact and is not uniformly equivalent to both the Bures metric  and the metric  induced by the C*-norm. 
\end{abstract}

\tableofcontents

\baselineskip=17pt

\section{Introduction}\label{s:intro}

The Bures metric \cite{Bures} of quantum information theory  is a metric specifically designed for quantum probability theory to detect how much two quantum probabilities/density operators are related \cite{Hayashi}.  In particular, if two density operators $a,b\in M_n(\C)$ are orthogonal in the sense that $\mathrm{Trace}(ab)=0$, then their distance is maximal in the Bures metric, and in fact, this orthogonality condition is equivalent to distance being maximal. Recently, Farenick and Rahaman \cite{Farenick-Rahaman17} advanced the theory of the Bures metric by generalizing it to the infinite-dimensional setting of unital C*-algebras equipped with faithful trace, which is the most general setting known to exist at the moment. We also note in this article (see Remark \ref{r:hellinger})  that this generalization of the Bures metric was more far-reaching than originally intended as it also generalized the Hellinger distance \cite{Hellinger} 
of probability theory, and thus providing a new noncommutative/quantum analogue of the Hellinger distance  that allows for the more abstract setting of C*-algebras as well as the ability to choose an arbitrary faithful trace  (for other quantum analogues to the Hellinger distance see \cite{Bhatia, Pitrik}). In \cite{Aguilar25}, the authors made further observations of the Bures metric in the setting of unital C*-algebras focused on the topological nature of the Bures metric. However, no mention has been made in the literature of the non-unital case of C*-algebras, which comprises of a large class of infinite-dimensional C*-algebras. Therefore, in this article, we show that the natural definition of the Bures metric in the non-unital case is in fact a metric. Moreover, as we now have a Bures metric for any C*-algebra equipped with a faithful trace, we are able and we do prove a Heine-Borel theorem for Bures density spaces in that we prove that the C*-algebra (unital or non-unital) is infinite dimensional if and only if the Bures density space is not compact (see Theorem \ref{t:hb-bures}). We also  provide explicit examples  of sequences that do not have converging subsequences in both commutative and noncommutative settings.

On the noncommutative metric geometry side of the story, quantum metric spaces of Rieffel were  introduced in \cite{Rieffel98a} motivated by noncommutative differential geometry in the sense of Connes \cite{Connes} and the need to introduce a more metric side of the noncommutative story to address statements in the high energy physics literature \cite{Rieffel00}. It's safe to say that  Rieffel's quantum metric spaces have been established as an important chapter in the story of quantum mathematics. Related recent work has also shown that seminorms from noncommutative geometry can be used to induce metric structures on completely positive maps in the C*-algebraic setting \cite{abelo25}, while in \cite{Aguilar25}, the authors used the quantum metric to induce a metric on the density space and made several topological observations. But again, just as with the Bures metric, the non-unital case has not been studied yet. However, Rieffel's compact quantum metric spaces cannot be used since they are, by definition, only for unital spaces. But thankfully, a non-unital notion of a quantum metric space has been introduced by Latr\'emoli\`ere \cite{Latremoliere12b} called a quantum locally compact metric space. Although this is the natural choice for the non-unital case, one issue does arise, which is that  a quantum metric can take value infinity (i.e. in general, it is at best an extended metric) (see Example \ref{e:extended} for an example of this). To remedy this, we introduce a subclass of quantum locally compact metric spaces called quantum Lipschitz triples, which have the benefit of a quantum metric that is truly a metric on the density space. We first prove that this is a natural setting since we show that the topology induced by the quantum metric for quantum Lipschitz triples is weaker than the topology induced by the C*-norm (see Theorem \ref{t:c*-norm-stronger-qm}). Next, unlike the Bures metric setting, the question of compactness seems to be more subtle, likely due to the dependence on the $L$-seminorm given by the quantum Lipschitz triple. But, we still perform a detailed analysis of matrix-valued functions the quantized interval, where we show that the Bures metric and quantum metric are not uniformly equivalent (Theorem \ref{t:non-uniform}) and that the quantum metric topology is not compact (Theorem \ref{t:non-compact}). As this example comes from a non-classical construction, we hope that it will provide future insights into the topological nature of the quantum metric on the density space for both the non-unital and  unital case. Overall, we hope this article not only begins the study of these fascinating metrics (the Bures metric and quantum metric) in the non-unital setting, but also provides foundational result for future work to build.

The article is organized as follows. In Section \ref{s:background}, we provide some background. In Section \ref{s:non-u-bures}, we define the Bures metric in the non-unital setting, while proving some preliminary results and providing some examples. In Section \ref{s:hb}, we prove the Heine-Borel theorem for Bures density spaces for both unital and non-unital C*-algebras. In Section \ref{s:qm-non-u}, we introduce the quantum metric on the density space in the non-unital case and the notion of a quantum Lipschitz triple. In Section \ref{s:quantized}, we focus on a family of examples that include noncommutative examples provided by matrix-valued functions on the quantized interval to analyze topological and metric comparisons with other metrics (the Bures metric and the metric induced by the C*-norm) and prove that the quantum metric topology is not compact in this case.

\section{Background}\label{s:background}

\begin{convention}
    Given a C*-algebra $\A$, we denote its unique C*-norm by $\|\cdot\|_\A$, its space of self-adjoint elements by $sa(\A)$, its cone of positive elements by $\A_+ $, its state space by $S(\A)$,   the spectrum of an element $a\in \A$ by $\sigma(a)$, and if $\A$ is unital, then we denote its unit by $1_\A$.
\end{convention}

As much of the techniques of this article rely on the standard minimal unitization of a C*-algebra, we define it here along with some usual associated objects and notation.

\begin{definition}[{\cite[Proposition I.1.3]{Davidson}}]\label{d:unitization}
    Let $\A$ be a non-unital C*-algebra. We denote the {\em standard minimal unitization of $\A$} by 
    \[
    \wt{\A}=\A\oplus \C,
    \]
    which is a unital C*-algebra with respect to  the unit $\un{\A}=(0,1)$ and coordinate-wise addition, scalar multiplication, and adjoint, and multiplication defined by 
    \[
    (a,\al)(b,\be)=(ab+\be a+\al b, \al\be),
    \]
    and C*-norm given by
    \[
    \|(a,\al)\|_{\wt{\A}}=\sup \{\|ab+\al b\|_\A : b\in \A, \|b\|_\A\leq 1\}.
    \]
    Note that $\A$ is *-isomorphic to the ideal $\{(a,0)\in \wt{\A}: a\in \A\}$, and we will often denote this ideal by $\A$ as well as any $(a,0)\in \wt{\A}$ by $a$.

    With this, we define the {\em spectrum of $a\in \A$} by
    \[
    \sigma(a)=\{\la \in \C : \la\in \sigma((a,0))\},
    \]
    that is $\la\in \sigma(a)$ if and only if $(a,0)-\la\un{\A}=(a,0)-\la(0,1)=(a,\la)$ is not invertible. In particular, we note that $0\in \sigma(a)$ (lest $\A$ be unital), and the continuous functional calculus of $a$ is given by 
    \[
    C^*(a)\cong C_0(\sigma(a)\setminus \{0\})
    \]
    (see \cite[Corollary I.3.2]{Davidson}).

    Finally, if $\rho $ is a positive linear functional on $\A$, we denote its unique extension to a positive linear functional on $\wt{\A}$ by $\wt{\rho}$ defined by
    \[
    \wt{\rho}((a, \al))=\rho(a)+\al\|\rho\|_\mathrm{op},
    \] 
    and so, when $\rho\in S(A)$, we have that $\wt{\rho}((a,\al))=\rho(a)+\al$. 
\end{definition}
 
To prepare for our work with quantum metrics, we begin with the definition of a compact quantum metric space in the sense of Rieffel since even in the non-unital case,  the  quantum metrics we work with in this article will be induced by compact quantum metrics. The definition of a compact quantum metric space takes on many forms and differs from Rieffel's initial definition, but the one we provide below is one of the main formulations. We first define. 
\begin{definition}[{\cite[Definition 2.3 and Remark 2.5 and 2.9]{Latremoliere12b}}]\label{d:l-pair}
    Let $\A$ be  a C*-algebra (unital or non-unital). Let $L:sa(\A) \rightarrow [0,\infty]$ be an extended seminorm (a seminorm allowed to take value $\infty$). We say that $(\A, L)$ is a {\em Lipschitz pair} if  $\dom{L}=\{a\in sa(\A): L(a)<\infty\}$  is a dense subspace of $sa(\A)$ and $\ker L=\{a\in sa(\A): L(a)=0\}=\R1_\A$ if $\A$ is unital, and $\ker L=\{a\in sa(\wt{\A}): L(a)=0\}=\R1_{\wt{\A}}$ if $\A$ is non-unital where we define $L((a,\la))=L(a)$ for every $(a,\la)\in sa(\wt{\A})$.

    We define the associated  {\em  Monge-Kantorovich metric} defined for each $\mu,\nu\in S(\A)$
        \[
        mk_L(\mu,\nu)=\sup \{|\mu(a)-\nu(a)| : a\in sa(\A), L(a)\leq 1\}.
        \] 
        We note that $mk_L$ may take on the value of  $\infty$, but it still satisfies all other properties of a metric, and thus, in general, $mk_L$ is an extended metric.
\end{definition}

\begin{definition}[{\cite{Rieffel98a}}]\label{d:cqms}
    Let $\A$ be a unital C*-algebra. Let $L:sa(\A)\rightarrow [0,\infty]$ be an extended seminorm. We say that the pair $(\A, L)$ is a {\em compact quantum metric space} if:
    \begin{enumerate}
        \item  $(\A, L)$ is a Lipschitz pair, and 
        \item the associated   Monge-Kantorovich metric $mk_L$ 
        metrizes the weak* topology on $S(\A)$.
    \end{enumerate}
    When this occurs, we call $L$ an $L$-seminorm.
\end{definition}

Now, we turn to the non-unital case. The definition of a quantum locally compact metric space in the sense of Latr\'emoli\`ere \cite{Latremoliere12b} is much more involved (for good reason), but we actually only need a part of the definition for the purposes of this article, and so, we only state this part. 

\begin{definition}[{\cite[Definition 2.27]{Latremoliere12b}}]\label{d:l-triple}
    Let $\A$ be a non-unital C*-algebra. Let $L:sa(\wt{\A})\rightarrow [0,\infty]$ be an extended seminorm. Let  $\M$ be a C*-subalgebra of $\A.$ 
    
    We say that $(\A, L,\M)$ is a {\em Lipschitz triple} if:
    \begin{enumerate}
        \item $(\wt{\A}, L)$ is a Lipschitz pair, and
        \item $\M$ is commutative and contains an approximate identity for $\A$.
    \end{enumerate}
    We will denote the restriction of $L$ to $\A$ also by $L$.
\end{definition}

\section{The Bures metric on the density space of a non-unital C*-algebra}\label{s:non-u-bures}

We extend the work of  \cite{Farenick-Rahaman17}, in which they extended the notion of the Bures metric from the finite-dimensional setting to that of any unital C*-algebra, by considering non-unital C*-algebras that include a large class of infinite-dimensional C*-algebras (a non-unital C*-algebra is necessarily infinite-dimensional since any finite-dimensional C*-algebra is unital (see \cite[Theorem III.1.1]{Davidson})). Along with proving that the Bures metric is in fact a metric in this non-unital setting, we also establish a Fuchs-van de Graaf Inequality  and comparisons with other topologies. To be clear about our definition of trace, we provide the following definition.

\begin{definition}
    Let $\A$ be a C*-algebra. We say that $\tau$ is a {\em trace on $\A$} if it is a positive linear functional on $\A$ satisfying $\tau(ab)=\tau(ba)$ for every $a,b\in \A$.
\end{definition}

\begin{definition}\label{d:density-Bures}
    Let $\A$ be a   C*-algebra (unital or non-unital) and let $\tau$ be a faithful trace. Define the {\em density space of $\A$ with respect to $\tau$} by
    \[
    D_\tau(\A)=\{ a\in \A_+: \tau(a)=1\}.
    \]

    Define the {\em Bures metric} on $D_\tau(\A)$ by
    \[
    d_B^\tau(x,y)=\sqrt{1- F_\tau(x,y)},
    \]
    where 
    \[
    F_\tau(x,y)=\tau(|\sqrt{x}\sqrt{y}|)
    \]
    is the {\em Fidelity   associated to $\tau$}.
\end{definition}

\begin{remark}\label{r:hellinger}
    As mentioned in the introduction, this generalization of the Bures distance to the C*-algebra setting (first done for unital C*-algebras in \cite{Farenick-Rahaman17}) actually also generalized the Hellinger distance \cite{Hellinger}. Indeed, if $X$ is a locally compact Hausdorff space equipped with a faithful measure $\mu$, and if we let $\mu$ also denote the faithful trace given by  integration with respect to $\mu$, then 
for any two positive $f,g\in C_0(X)$ such that $\mu(f)=1=\mu(g)$, we have that since $C_0(X)$ is commutative
\[
d_B^\mu(f,g)=\sqrt{1-\int_X \sqrt{fg} \ d \mu}=\sqrt{\frac{1}{2}\int_X \left( \sqrt{f}-\sqrt{g}\right)^2 \ d\mu}.
\]
    We thank Katy Craig at UC Santa Barbara for making this observation during a talk the first author gave at the UC Santa Barbara Applied Math/PDE/Data Science seminar in 2026.
\end{remark}

We first prove some basic facts about some operations regarding the unitization, which are  likely well-known.
\begin{lemma}\label{l:unitize-root-abs}
    Let $\A$ be a non-unital C*-algebra. Let $x\in \A$. It holds that:
    \begin{enumerate}
        \item if $x\in \A_+$, then $(x, 0)\in \wt{\A}_+ $, $(\sqrt{x}, 0)\in \wt{\A}_+$, and $\sqrt{(x,0)}=\left( \sqrt{x}, 0\right)$;
        \item $|(x,0)|=(|x|,0)$.
    \end{enumerate}
\end{lemma}
\begin{proof}
    (1) Since $x\in \A_+$, we have that $x=y^*y$ for some $y\in \A$. But $(x,0)=(y^*y,0)=(y^*,0)(y,0)=(y,0)^*(y,0)\in \wt{\A}_+$. Hence as $\sqrt{x}\in \A_+$, we have that $(\sqrt{x},0)\in \wt{\A}_+$. Furthermore, 
    \[
    (\sqrt{x},0)^2=(\sqrt{x}^2,0)=(x,0),
    \]
    and thus $\sqrt{(x,0)}=(\sqrt{x},0)$ by uniqueness of square root.

    (2) By part (1), we have that 
    \[
    (|x|, 0)=\left( \sqrt{x^*x},0\right)=\sqrt{(x^*x, 0)}=\sqrt{(x^*,0)(x,0)}=|(x,0)|
    \]
    as desired.
\end{proof}

Let's now see how the standard minimal unitization (Definition \ref{d:unitization}) will begin to help us prove our results.

\begin{lemma}\label{l:bures-unitization}
     Let $\A$ be a non-unital C*-algebra and let $\tau$ be a faithful trace. For every $x,y \in D_\tau(\A)$ it holds that
     \[
     F_{\wt{\tau}}((x,0),(y,0))=F_\tau(x,y)
     \]
     and 
     \[
     d_B^{\wt{\tau}}((x,0), (y,0))=d_B^\tau(x,y),
     \]
     where $\wt{\tau}$ is the unique extension of $\tau $ to the unitization $\wt{\A}$ (Definition \ref{d:unitization}).
\end{lemma}
\begin{proof}
    Let $x,y\in D_\tau(\A)$. We have  by Lemma \ref{l:unitize-root-abs} that $(x,0),(y,0)\in \wt{\A}$. Also $\wt{\tau}((x,0))=\tau(x)=1=\tau(y)=\wt{\tau}((y,0))$.  Hence, $(x,0),(y,0)\in D_{\wt{\tau}}\left( \wt{\A}\right)$. Therefore, the quantities $ F_{\wt{\tau}}((x,0),(y,0))$ and    $d_B^{\wt{\tau}}((x,0), (y,0)) $ are well-defined by \cite[Theorem 2.6]{Farenick-Rahaman17}. Moreover, by Lemma \ref{l:unitize-root-abs},
    \begin{align*}
F_{\wt{\tau}}((x,0),(y,0))& = \wt{\tau}\left(\left|\sqrt{(x,0)}\sqrt{(y,0)}\right|\right)\\& = \wt{\tau}\left(\left|\left(\sqrt{x},0 \right)\left(\sqrt{y},0 \right) \right| \right)\\
& = \wt{\tau}\left(\left|\left(\sqrt{x} \sqrt{y},0 \right) \right| \right)\\
& = \wt{\tau}\left(\left(\left|\sqrt{x} \sqrt{y}\right| ,0 \right) \right)\\
& = \tau\left(\left|\sqrt{x} \sqrt{y}\right|\right)+0\|\tau\|_\mathrm{op}\\
& =  \tau\left(\left|\sqrt{x} \sqrt{y}\right|\right)\\
& = F_\tau(x,y).
    \end{align*}
    The other equality follows by definition.
\end{proof}
 
We need one more fact before we prove that the Bures metric in the non-unital case is, in fact, a metric. The main issue is the coincidence property (that is, $d_B^\tau(x,y)=0\iff x=y$). To accomplish this, we first need to establish a fact about extending faithful positive linear functionals. We suspect the following result is well-known, but it seems to be difficult to find in the literature. We thank Jack Spielberg for sharing the following proof with us.

\begin{proposition}\label{p:faithful}
    Let $\A$ be a non-unital C*-algebra. If $\rho$ is a faithful positive linear functional  on $\A$, then its unique extension to a positive linear functional on  $\wt{\A}$, $\wt{\rho}$, is faithful.
\end{proposition}
\begin{proof}
    Up to normalization, we may assume that $\rho\in S(\A)$. By \cite[Proposition 2.1.5(vi)]{Dixmier-c-book}, we have that $\wt{\rho}\in S\left(\wt{\A}\right)$.

Set $N=\{x\in \wt{A}: \wt{\rho}(x^*x)=0\}$, which is a left ideal in $\wt{\A}$. If we show that $N=\{0\}$, or equivalently that $\dim(N)=0$, then the proof is complete.

We first show that $\A\cap N=\{0\}$. Assume by way of contradiction that there exists $(a,\al)\in \A\cap N$ such that $(a,\al)\neq \{0\}$. Since $(a,\al)\in \A$, we have that $\al=0$, and so $a\neq 0$. Thus, as $(a,\al)\in N$, we have $0=\wt{\rho}((a,0)^*(a,0))=\wt{\rho}((a^*a,0))=\rho(a^*a)$. But as $a\neq 0$ and $\rho$ is faithful, we have reached a contradiction. Thus $\A\cap N=\{0\}$.

Next, we claim that $\dim(N)\in \{0,1\}$. Assume by way of contradiction that $\dim(N)\geq 2$. Thus there exist non-zero  $x_j=(a_j,\la_j)\in N$ for $j\in \{1,2\}$ that are linearly independent. Since $\A\cap N=\{0\}$, we have that $\la_j\neq 0$ for $j\in \{1,2\}$. By linear independence, we have that $y=\la_1^{-1}x_1-\la_2^{-1}x_2 \neq 0$ and $y\in N$ as $N$ is a subspace. However,
\begin{align*}
    y & = \la_1^{-1}x_1-\la_2^{-1}x_2\\
    &= (\la_1^{-1}a_1, 1)-(\la_2^{-1}a_2, 1)\\
    & = (\la_1^{-1}a_1-\la_2^{-1}a_2, 0)\in \A.
\end{align*}
Thus $y\in \A\cap N$, but since $\A\cap N=\{0\}$, we have reached a contradiction, and so $\dim(N)\in \{0,1\}$.

Next, assume by way of contradiction that $\dim(N)=1$. Thus, $N$ has some non-zero element $z=(c,\la)$. Since $\A\cap N=\{0\}$, we have that $\la \neq 0$, and we set $a=\la^{-1}c\in \A$ and set $x=\la^{-1}(c,\la)=(a,1)\in N$.  Set $e=-a$, and note that $x=(-e,1)$  Let $b\in \A$. Then
\[
bx=(b,0)(-e,1)=(-be+b, 0)\in \A
\]
and since $N$ is a left   ideal, we have that $bx\in N$, and so $bx=0$ since $\A\cap N=\{0\}$, which implies that $-be+b=0$. Thus $b=be$.  Hence $e$ is a right identity in $\A$. We now show that $e$ is also a left identity. Consider $a^*x$, which is an element of $N$ since $N$ is a left ideal. However, $a^*x\in \A$ since $\A$ is an ideal. Therefore, $a^*x=0$ since $\A\cap N=\{0\}$. However, 
\[
0=a^*x=(a^*,0)(a,1)=(a^*a+a^*,0),
\]
and so $a^*a+a^*=0$, which implies that $-a^*a=a^*$. Hence,  $a^*$ is self-adjoint, and consequently, so are $a$ and $e$. Let $b\in \A$. Then since $e$ is a right identity, and we have that $b^*e=b^*$. Furthermore,  $e$ is self-adjoint, and we have 
\[
b=(b^*)^*=(b^*e)^*=e^*(b^*)^*=eb.
\] Therefore, $e\in \A$ is both a left and right identity, which contradicts that $\A$ is non-unital.  Thus, $\dim(N)=0$, which completes the proof.
\end{proof}

We now prove that the Bures metric is in fact a metric on the density space associated to any non-unital C*-algebra equipped with a faithful trace.

\begin{theorem}\label{t:bures-metric-non-unital}
    Let $\A$ be a non-unital C*-algebra equipped with a faithful trace $\tau$. It holds that the Bures metric $d_B^\tau$ is a metric on the density space $D_\tau(\A)$.
\end{theorem}
\begin{proof}
 Symmetry and the triangle inequality follow from Lemma \ref{l:bures-unitization} and \cite[Theorem 2.6]{Farenick-Rahaman17}. All that remains is the coincidence property. Let $x,y \in D_\tau(\A)$. If $x=y$, then $d_B^\tau(x,y)=0$ follows since $\tau(|\sqrt{x}\sqrt{x}|)=\tau(|x|)=\tau(x)=1$.  Now, assume that $d_B^\tau(x,y)=0$. Thus 
 \[
 d_B^{\wt{\tau}}((x,0), (y,0))=d_B^\tau(x,y)=0
 \]
 by Lemma \ref{l:bures-unitization}. 
 Then as $\wt{\tau}$ is faithful by Proposition \ref{p:faithful}, and so $d_B^{\wt{\tau}}$ is a metric on $\wt{\A}$ by \cite[Theorem 2.6]{Farenick-Rahaman17}, we have that 
 \[
 (x,0)= (y,0),
 \]
and so $x=y.$
\end{proof}

The first topology we compare the Bures metric topology to is the $L^1$-topology given by the faithful trace. Just as in \cite{Farenick-Rahaman17}, we first check that we get an $L^1$-topology to begin with.

\begin{theorem-definition}\label{p:l1-trace-topology}
    Let $\A$ be a non-unital C*-algebra equipped with a faithful trace $\tau$. It holds that the map
    \[
    a\in \A \longmapsto \tau(|a|)
    \]
    is a norm on $\A$ denoted $\|\cdot \|_{\tau,1}$ and called the {\em $L^1$-norm}.  We denote its induced metric by $d_1^\tau$, called the {\em $L^1$-metric}, and we call the associated topology, {\em the $L^1$-topology}.
\end{theorem-definition}
\begin{proof}
    By Proposition \ref{p:faithful}, we have that $\wt{\tau}$ is a faithful trace on the unitization $\wt{\A}$.  Thus $\|\cdot\|_{\wt{\tau},1}$ is a norm on $\wt{\A}$ by \cite[Proposition 2.1]{Farenick-Rahaman17}. Furthermore, if $a\in \A$, then by Lemma \ref{l:unitize-root-abs}
    \begin{align*}
        \|(a,0)\|_{\wt{\tau},1}& = \wt{\tau}(|(a,0)|)\\
        &= \wt{\tau}((|a|,0))\\
        & =\tau(|a|)\\
        &= \|a\|_{\tau,1}.
    \end{align*}
    Hence, $\|\cdot\|_{\tau,1}$ is a norm on $\A$.
\end{proof}

We now prove that the topology induced by the Bures metric and the $L^1$-metric are equal following \cite[Proposition 2.7]{Farenick-Rahaman17} by providing a Fuchs-van de Graaf inequality in the non-unital case.

\begin{theorem}\label{t:l1-homeo}
    Let $\A$ be a non-unital C*-algebra equipped with a faithful trace $\tau$. For every $x,y\in D_\tau(\A)$, it holds that 
    \[
    1-\frac{1}{2}d_1^\tau(x,y)\leq F_\tau(x,y) \leq \sqrt{1-\frac{1}{4}d_1^\tau( x,y)^2}, 
    \]
    and equivalently,
    \[
    2d_B^\tau(x,y)^2\leq d_1^\tau(x,y)\leq 2\sqrt{1-F_\tau(x,y)^2}.
    \]
    Moreover, the topologies induced by $d_B^\tau$ and $d_1^\tau$ are equal.
\end{theorem}
\begin{proof}
    We have
    \begin{align*}
         1-\frac{1}{2}d_1^\tau(x,y)& =  1-\frac{1}{2}\|x-y\|_{\tau,1}\\
         & = 1-\frac{1}{2}\|(x-y, 0)\|_{\wt{\tau},1} \quad \text{ by proof of Theorem-Definition \ref{p:l1-trace-topology}}\\
         & = 1-\frac{1}{2}\|(x, 0)-(y, 0)\|_{\wt{\tau},1}\\
         & \leq F_{\wt{\tau}}((x,0),(y,0)) \quad \text{ by \cite[Proposition 2.7]{Farenick-Rahaman17}}\\
         & =F_\tau(x,y) \quad \text{ by Lemma \ref{l:bures-unitization}}\\
         & \leq \sqrt{1-\frac{1}{4}d_1^{\wt{\tau}}( (x,0),(y,0))^2} \quad \text{ by \cite[Proposition 2.7]{Farenick-Rahaman17}}\\
         & = \sqrt{1-\frac{1}{4}d_1^\tau( x,y)^2} \quad \text{ by proof of Theorem-Definition \ref{p:l1-trace-topology}}.
    \end{align*}

    The equivalent inequalities follow  from arithmetic computations.

    Next, let $(x_n)_{n\in \N}$ be a sequence in $D_\tau(\A)$ and let $x\in D_\tau(\A)$.  First, assume that $(x_n)_{n\in \N}$ converges to $x$ with respect to $d_1^\tau$. Thus,  $((d_1^\tau(x_n,x)))_{n\in \N}$ converges to $0$, and so $((d_B^\tau(x_n,x)))_{n\in \N}$ converges to $0$ since  $2d_B^\tau(x_n,x)^2\leq d_1^\tau(x_n,x) $ for every $n\in \N$. Hence,  $(x_n)_{n\in \N}$ converges to $x$ with respect to $d_B^\tau$.

    Next, suppose  $(x_n)_{n\in \N}$ converge to $x$ with respect to $d_B^\tau$. Thus,  $((d_B^\tau(x_n,x)))_{n\in \N}$ converges to $0$. Now, for every $n\in \N$, we have that 
    \[
    F_\tau(x_n,x)=1-d_B^\tau(x_n,x)^2.
    \]
    Hence, $((F_\tau(x_n,x)))_{n\in \N}$ converges to $1$, and thus 
    \[
    \left( \left( 2\sqrt{1-F_\tau(x_n,x)}\right)\right)_{n\in \N}
    \]
    converges to $0$. As a result,  $((d_1^\tau(x_n,x)))_{n\in \N}$ converges to $0$ since  \[d_1^\tau(x_n,x)\leq 2\sqrt{1-F_\tau(x_n,x)}\] for every $n\in \N$. Thus $(x_n)_{n\in \N}$ converges to $x$ with respect to $d_1^\tau$, which implies that the topologies coincide.
\end{proof}

We now compare the Bures metric topology with the topology induced by the C*-norm on the density space.

\begin{definition}\label{d:c*-norm}
    Let $\A$ be a C*-algebra. We denote the metric induced by the C*-norm by $d_\A$, called the {\em C*-norm metric}. That is, for every $x,y\in \A$
    \[
    d_\A(x,y)=\|x-y\|_\A.
    \]
\end{definition}

\begin{theorem}\label{t:c*-norm-bures}
    Let $\A$ be a non-unital C*-algebra equipped with a faithful trace $\tau$. On the density space $D_\tau(\A)$, it holds that the topology induced by C*-norm metric $d_\A$ is finer than the  topology induced by the Bures metric $d_B^\tau$.
\end{theorem}
\begin{proof}
Since the $L^1$-topology and the Bures metric topology coincide on $D_\tau(\A)$, we just need to show that the $d_\A$ topology is finer than the $L^1$-topology. Let  $x\in \A$, then
\begin{align*}
\tau(|x|)^2 & \leq \|\tau\|_\mathrm{op}^2 \||x|\|_\A^2=\|\tau\|_{op}^2 \||x|^*|x|\|_\A=\|\tau\|_\mathrm{op}^2 \||x||x|\|_\A\\
&=\|\tau\|_\mathrm{op}^2 \|x^*x\|_\A=\|\tau\|_\mathrm{op}^2 \|x\|^2_\A,
\end{align*}
and so
\[
\tau(|x|)\leq \|\tau\|_\mathrm{op} \|x\|_\A.
\]
Hence convergence in $d_\A$ implies convergence in the $L^1$-topology, which completes the proof by the above comments.
\end{proof}

In general, the C*-norm metric topology need not equal the Bures metric topology since the latter topology is the $L^1$-topology. However, in order to provide a finer analysis of the behavior of the Bures metric, we provide an explicit sequence in a certain non-unital C*-algebra that converges with respect to the Bures metric, but not with respect to the C*-norm metric. Afterwards, we generalize the following example to some noncommutative cases.

\begin{example}
Consider the non-unital commutative C*-algebra $C_0((0,1])$. For each $x\in (0,1]$ and $n\in \N$, define
\[ f(x) = 
\begin{cases} 
      \frac{18}{5} x & x \in (0,\frac{1}{3}] \\
      \frac{6}{5} & x \in (\frac{1}{3}, 1]
   \end{cases}
\]
and 
\[ f_n(x) =
\begin{cases} 
      \frac{18 \cdot 2^n}{5 \cdot 2^n - 3} x & x \in (0,\frac{1}{3}] \\
      \frac{6 \cdot 2^n}{5 \cdot 2^n - 3} & x \in (\frac{1}{3}, 1 - 2^{-n}) \\
      - \frac{6 \cdot 4^n}{5 \cdot 2^n - 3} (x-1)  & x \in [1 - 2^{-n}, 1] 
   \end{cases}
\]
(see Figure \ref{f:norm-bures}).
\begin{figure}[h]
    \centering
\includegraphics[width=8cm]{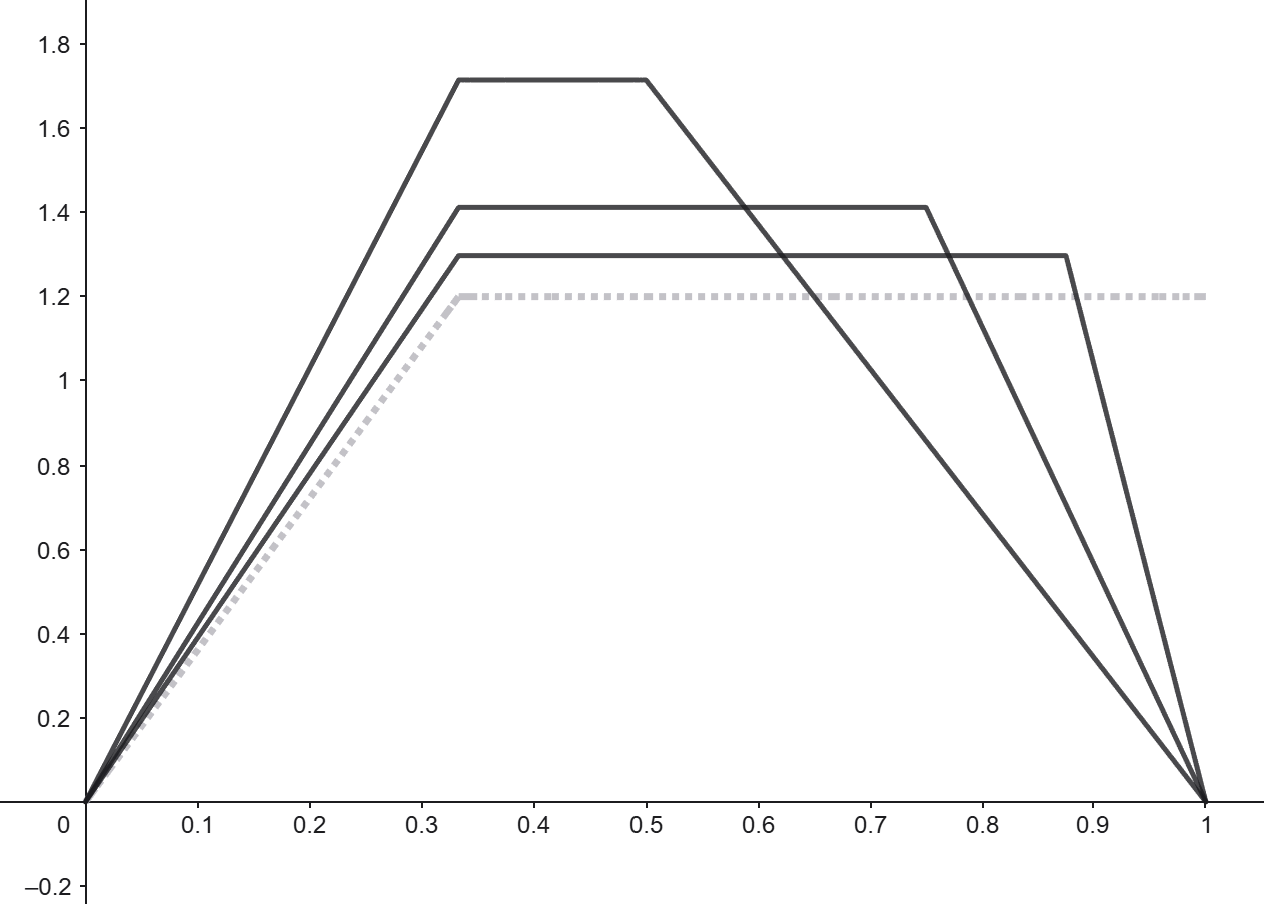}
    \caption{$f$(dotted), $f_1$, $f_2$, and $f_3$. Plotted in GeoGebra}
    \label{f:norm-bures}
\end{figure}
Consider the faithful trace $\lambda$ on $C_0((0,1])$ given by Lebesgue integration over $(0,1]$. It follows  that  $f\in D_\lambda(C_0((0,1])))$ and $f_n\in D_\lambda(C_0((0,1]))) $ for every $n\in \N$. Next, we calculate the Bures metric to get a better idea of the behavior.

Let $n\in \N$ and  $x\in (0,1]$, we then have
\[
\left|\sqrt{f_n}\sqrt{f}\right|(x)=\begin{cases}
 \frac{18\cdot 2^{n/2}}{\sqrt{25\cdot 2^n-15}}x    &x\in (0,\frac{1}{3}]\\
 \frac{6\cdot 2^{n/2}}{\sqrt{25\cdot 2^n-15}} & x\in (\frac{1}{3}, 1-2^{-n})\\
 \frac{6\cdot 4^{n/2}}{\sqrt{25\cdot 2^n-15}} \sqrt{1-x} & x\in [1-2^{-n}, 1]
\end{cases}
\]
\begin{figure}[h]
\includegraphics[width=8cm]{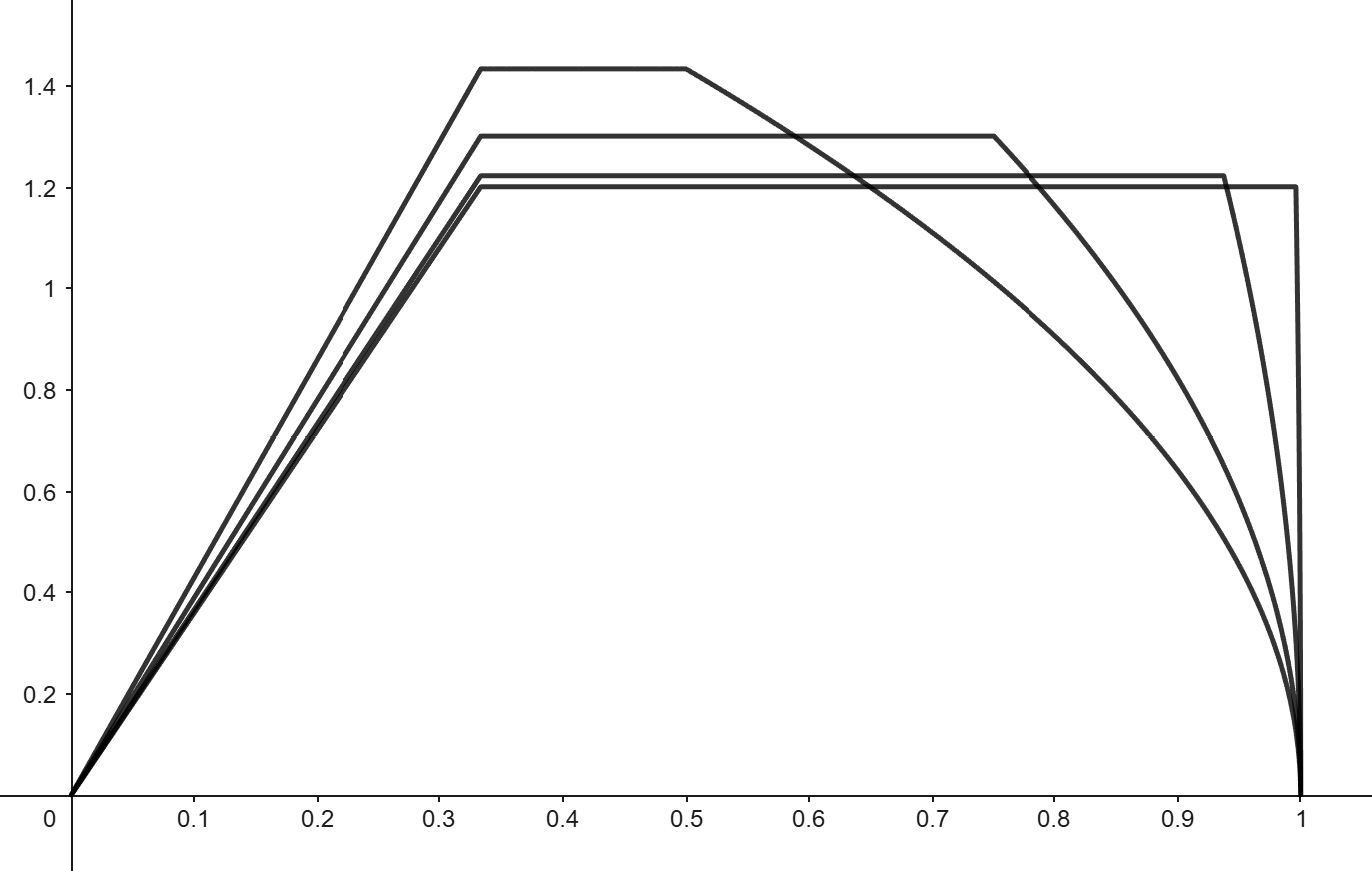}
\caption{$\left|\sqrt{f_1}\sqrt{f}\right|, \left|\sqrt{f_2}\sqrt{f}\right|, \left|\sqrt{f_4}\sqrt{f}\right|, \left|\sqrt{f_8}\sqrt{f}\right|$. Plotted in GeoGebra.}
    \label{f:fid}
\end{figure}
In Figure \ref{f:fid}, we see this sequence of functions tend toward this more trapezoid-like shape with area $1$, which is desirable because we want the integral of this sequence of functions to tend to $1$ since this value is being subtracted from $1$ in the definition of the Bures metric. Formally, one can calculate for each $n\in \N$ the fidelity
\[
F_\lambda(f_n,f)= \frac{  2^{n/2}}{\sqrt{25\cdot 2^n-15}}+\frac{2^{1-(n/2)}\left(2^{n+1}-3\right)}{\sqrt{25\cdot 2^n-15}}+\frac{2^{n+2}\sqrt{8^{-n}}}{\sqrt{25\cdot 2^n-15}},
\]
and
\[
\lim_{n\to \infty} F_\lambda(f_n,f)=\frac{1}{5}+\frac{4}{5}+0=1,
\]
and so
\[
\lim_{n\to \infty} d_B^\lambda(f_n,f)=\sqrt{1-F_\lambda(f_n,f)}=0.
\]
Thus $(f_n)_{n\in \N}$ converges to $f$ with respect to $d_B^\lambda$. Moreover, $(f_n)_{n\in \N}$ does not converge to $f$ with respect to the C*-norm since it does not converge uniformly as it converges pointwise to a discontinuous function.
\end{example}

To extend this example to some noncommutative cases,  we state a well-known construction.

\begin{proposition}\label{p:faithful-extend}
    Let $X$ be a locally compact Hausdorff space equipped with a  faithful measure $\mu$ such that $\mu(X)<\infty$. Let $\A$ be a C*-algebra equipped with a faithful trace $\rho$. 

    If we define
    \[
    \tau_\rho^\mu(a)=\int_X \tau(a(x))\ d\mu(x)
    \]
    for every $a\in C_0(X,\A)$, 
    then $\tau_\rho^\mu$ is a faithful trace on $C_0(X,\A)$.
\end{proposition}
\begin{proof}
    First $\tau^\mu_\rho$ is positive and linear since $\tau$ and integration are both positive and linear. Since $\tau$ is a trace, we have that $\tau^\rho_\mu$ is a trace. Finally, let $a\in C_0(X, \A)$ be positive such that $a\neq 0$. Hence there exists $y\in X$ such that $a(y)\neq 0$. Thus $\tau(a(y))>0$. Now $\tau\circ a$ is continuous on $X$, and so there exists an open neighborhood $U$ of $y$ such that there exists $\tau \circ a(x)$ is greater than $\tau(a(y))/2>0$ for every $x\in U$. Therefore, as $\mu$ is a faithful measure and $\tau\circ a$ is a non-negative functions, we have that $\mu(U)>0 $ and $\tau_\rho^\mu(a)\geq \mu(U) \cdot \tau(a(y))/2>0$. Hence $\tau_\rho^\mu$ is faithful.
\end{proof}

\begin{example}\label{e:noncomm-nonunital}
    Let $X$ be a locally compact Hausdorff space equipped with a faithful measure $\mu$, and as an abuse of notation, denote its associated faithful trace by $\mu$ as well. Let $\A$ be a unital C*-algebra equipped with a faithful trace $\rho$. Consider the non-unital C*-algebra $C_0(X,\A)$ and  the faithful trace
    \[
    \tau_\rho^\mu
    \]
    of Proposition \ref{p:faithful-extend}.

    Assume there exists a sequence $(g_n)_{n\in \N}$ in $D_\mu(C_0(X))$ and a $g\in D_\mu(C_0(X))$ such that $(g_n)_{n\in \N}$   converges to $g$ with respect to $d_B^\mu$ but does not converge uniformly to $g$.

     Next, for each $x\in X$ and $n\in \N$, define
\[ g^\A(x) = g(x)\cdot \frac{1}{\rho(1_\A)}1_\A 
\]
and 
\[ g_n^\A(x) =g_n(x)\cdot \frac{1}{\rho(1_\A)}1_\A .
\]  By linearity of $\rho$
\begin{align*}
\tau_\rho^\mu(g^\A)&=\int_{X} \rho(g^\A(x)) \ d\mu(x)\\
& = \int_{X} \rho\left(  g(x)\cdot \frac{1}{\rho(1_\A)}1_\A \right) \ d\mu(x)\\
& = \int_{X}  g(x)\cdot \frac{1}{\rho(1_\A)}\rho(1_\A)  \ d\mu(x)\\
& = \int_{X} g(x) \ d\mu(x)\\
& = \mu(g).
\end{align*}
Similarly, for every $n\in\N$,
\[
\tau_\rho^\mu(g_n^\A)=\tau(g_n) \quad \text{ and } \quad d_B^{\tau_\rho^\mu}(g^\A_n, g^\A)=d_B^\mu(g_n,g).
\]
Furthermore, note  that \[\|g^\A_n-g^\A\|_{C_0(X,\A)}=\frac{1}{\rho(1_\A)}\|g_n-g\|_{C_0(X)}\]
for every $n\in \N$.

Hence, $(g^\A_n)_{n\in \N}$ is a sequence in $D_{\tau_\rho^\mu}(C_0(X, \A))$ that converges to \[g^\A\in D_{\tau_\rho^\mu}(C_0(X, \A))\] with respect to $d_B^{\tau_\rho^\mu}$ but  $(g^\A_n)_{n\in \N}$ does not converge to $g^\A$ with respect to the C*-norm.

One can use this to come up with many noncommutative examples. For instance, we get a noncommutative example with $X=(0,1]$, $\mu=\la$ (the Lebesgue measure), $\A=M_n(\C)$, $\rho=\mathrm{Tr} $ (the usual trace on matrices), and $g=f$ and $g_n=f_n$ for every $n\in \N$ from Example \ref{e:noncomm-nonunital}.
\end{example}

\subsection{A Heine-Borel theorem for Bures density spaces}\label{s:hb}

In this section, we prove that when we have a C*-algebra $\A$ (unital or non-unital) equipped with a faithful trace $\tau$, then the Bures metric space $(D_\tau(\A), d_B^\tau)$ is not compact if and only if $\A$ is infinite-dimensional. The forward  direction was already established \cite[Proposition 3.3 and Theorem 3.5]{Aguilar25}, so the purpose of this section is to establish the reverse direction in not only the unital case, but also the non-unital case thanks to the results of the last section. We also provide various explicit examples of sequences that have no convergent subsequences in the infinite-dimensional setting.

We first prove a fact about the $L^1$-norm of Theorem-Definition \ref{p:l1-trace-topology}.

\begin{proposition}\label{p:l1-star-invariant}
    Let $\A$ be a C*-algebra (unital or non-unital). Let $\tau$ be a faithful trace on $\A$. It holds that 
    \[
    \|x\|_{\tau,1}=\|x^*\|_{\tau,1}
    \]
    for every $x\in \A$.
\end{proposition}
\begin{proof}
Let $(q_n)_{n\in \N}$ be a sequence of polynomials  that converges uniformly to the square root function on $[0,\|x\|^2_\A]$, which exists by the  Stone-Weierstra\ss{}  theorem. In order to ensure that the non-unital case also works, we define $p_n(r)=q_n(r)-q_n(0)$ for every $n\in \N, r\in \R$. Thus, for every $n\in \N$,  $p_n(0)=0$,  and so $p_n\in C_0(\sigma(x^*x)\setminus \{0\})$, and $(p_n)_{n\in \N}$ converges uniformly to the square root function on $[0,\|x\|^2_\A]$. Thus, by functional calculus \cite[Corollary I.3.2]{Davidson}, we have that $(p_n(x^*x))_{n\in \N}$ converges to $\sqrt{x^*x}=|x|$ with respect to $\|\cdot\|_\A$. Therefore, by continuity of positive linear functionals \cite[Lemma I.9.5]{Davidson}, we have that $(\tau(p_n(x^*x)))_{n\in \N}$ converges to $\tau(|x|). $ However, since $\tau$ is a trace, we have that 
\[
\tau(p_n(x^*x))=\tau(p_n(xx^*))
\]
for every $n\in \N$, and by functional calculus again, we have that $(p_n(xx^*))_{n\in \N}$ converges to $\sqrt{xx^*}=|x^*|$. Therefore, by continuity of $\tau$, we have that $(\tau(p_n(xx^*)))_{n\in \N}$ converges to $\tau(|x^*|)$. Consequently, by uniqueness of limits we have that
\[
\tau(|x|)=\tau(|x^*|)
\]
as desired.
\end{proof}
And now, the Heine-Borel theorem for Bures density spaces.
\begin{theorem}\label{t:hb-bures}
    Let $\A$ be a C*-algebra (unital or non-unital) equipped with a faithful trace $\tau$. The following are equivalent:
    \begin{enumerate}
        \item the Bures density space $(D_\tau(\A), d_B^\tau)$ is not compact;
        \item $\A$ is infinite dimensional.
    \end{enumerate}
\end{theorem}
\begin{proof}
    (1)$\implies$(2) is given by \cite[Proposition 3.3 and Theorem 3.5]{Aguilar25}.

    (2)$\implies$(1) We proceed by contrapositive. Suppose  that the Bures density space $(D_\tau(\A), d_B^\tau)$ is compact. As  a result,  $(D_\tau(\A), d_1^\tau)$ is compact by \cite[Corollary 2.8]{Farenick-Rahaman17} for the unital case or by Theorem \ref{t:l1-homeo} for the non-unital case. Recall that $(\A, \|\cdot\|_{\tau,1})$ is a normed vector space by \cite[Proposition 2.1]{Farenick-Rahaman17} for the unital case or by \ref{p:l1-trace-topology} for the non-unital case, and  the topology induced on $D_\tau(\A)$ by $ d_1^\tau$ is just the subspace topology induced by $\|\cdot\|_{\tau,1}$ on $D_\tau(\A)$   as a subset of $\A$.  Let $(x_n)_{n\in \N}$ be a sequence in 
    \[
    Ball(\A, \|\cdot\|_{\tau,1})=\{a \in \A: \|a\|_{\tau,1}\leq 1\}.
    \]
    We will show that $(x_n)_{n\in \N}$ has a converging subsequence with respect to the $L^1$-metric $\|\cdot\|_{\tau,1}$ whose limit is in $Ball(\A, \|\cdot\|_{\tau,1})$. 
    
    If $(x_n)_{n\in \N}$ had infinitely many zero terms, then it would have a converging subsequence in any metric. Hence, we assume that $(x_n)_{n\in \N}$ has only finitely many zero terms, in which case we can reduce to a subsequence that has no zero terms, and therefore, we assume for the remainder of the proof that $(x_n)_{n\in \N}$ has no zero terms.

    Let $n\in \N$. Consider
    \[
    Re(x_n)=\frac{1}{2}(x_n+x_n^*) \quad \text{ and } \quad Im(x_n)=\frac{1}{2i}(x_n-x_n^*).
    \]
   It holds that 
    \[
    \|Re(x_n)\|_{\tau,1}\leq\frac{1}{2}(\|x_n\|_{\tau,1}+\|x_n^*\|_{\tau,1})= \frac{1}{2}(\|x_n\|_{\tau,1}+\|x_n\|_{\tau,1})\leq 1
    \]
    by Proposition \ref{p:l1-star-invariant}. Similarly $\|Im(x_n)\|_{\tau,1}\leq 1$.

    Next, note that 
    \[
    Re(x_n)=Re(x_n)_+-Re(x_n)_-,
    \]
    where $Re(x_n)_+,Re(x_n)_-\in \A_+$ by \cite[Corollary I.4.2]{Davidson} as $Re(x_n)$ is self-adjoint. Now set \[y_n=Re(x_n)_+.\]
    Similarly, as above, we can either have a subsequence of zero terms or produce a subsequence that has no zero terms, and so we assume for the remainder of the proof that $(y_n)_{n\in \N}$ has no zero terms. 
By \cite[Corollary I.4.2]{Davidson}, we have that 
\[
Re(x_n)_+=\frac{1}{2}(Re(x_n)+|Re(x_n)|).
\]
Therefore, since the absolute value of a positive element is itself, we have that 
\[
\|y_n\|_{\tau,1}\leq \frac{1}{2}(\|Re(x_n)\|_{\tau,1}+\||Re(x_n)|\|_{\tau,1})=\frac{1}{2}(\|Re(x_n)\|_{\tau,1}+\|Re(x_n)\|_{\tau,1})\leq 1.
\]
Thus, as $[0,1]$ is compact, we have that $(\|y_n\|_{\tau,1})_{n \in \N}$ has a converging subsequence $(\|y_{n_m}\|_{\tau,1})_{m\in \N}$ that converges to some $\alpha \geq 0$.  Now, as $\tau$ is faithful, we have that $\|y_{n_m}\|_{\tau,1}>0$ for every $m \in \N$ since $y_{n} \neq 0$ for every $n\in \N$. Thus for each $m\in \N$, we may define
\[
z_{n_m}=\frac{1}{\|y_{n_m}\|_{\tau,1}}y_{n_m}.
\]
By construction, we have that $z_{n_m}\in \A_+$ and $\tau(z_{n_m})=1$ for every $m\in \N$, and thus $(z_{n_m})_{m\in \N}$ is a sequence in $D_\tau(\A)$. Therefore, $(z_{n_m})_{m\in \N}$ has a converging subsequence $(z_{n_{m_j}})_{j\in \N}$  that converges to some $b\in D_\tau(\A)$ with respect to $\|\cdot\|_{\tau,1}$. Hence, the sequence 
\[
(\|y_{n_{m_j}}\|_{\tau,1}z_{n_{m_j}})_{j\in \N}
\]
converges to $\alpha b$ with respect to $\|\cdot\|_{\tau,1}$. But 
\[
y_{n_{m_j}}=\|y_{n_{m_j}}\|_{\tau,1}z_{n_{m_j}}
\]
for every $j\in \N$. Similarly, we can find a subsequence of $(Re(x_{n_{m_j}})_-)_{j\in \N}$ that converges, and then continue this process for a further  subsequence of  $(Im(x_n)_+)_{n\in \N}$ and then a  further subsequence of $(Im(x_n)_-)_{n\in \N}$ to find an subsequential indexing $(N_k)_{k\in \N}$ such that 
\[
(Re(x_{N_k})_+)_{k\in \N}, (Re(x_{N_k})_-)_{k\in \N}, (Im(x_{N_k})_+)_{k\in \N}, (Im(x_{N_k})_-)_{k\in \N}
\]
all converge, and thus since   $x_{N_k}$ is just a linear combination of these terms for each $k\in \N$, we have that $(x_{N_k})_{k\in \N}$ converges to some $c\in \A$ with respect to $\|\cdot\|_{\tau,1}$. Since $ Ball(\A, \|\cdot\|_{\tau,1})$ is closed with respect to $\|\cdot\|_{\tau,1}$, we have that  $c\in Ball(\A, \|\cdot\|_{\tau,1})$. Thus, $ Ball(\A, \|\cdot\|_{\tau,1})$ is compact with respect to $\|\cdot\|_{\tau,1}$. Therefore, $\A$ is finite dimensional by the Heine-Borel theorem.
\end{proof}
We focus particularly in the non-unital case for the next result.
\begin{corollary}
  Let $\A$ be a C*-algebra equipped with a faithful trace $\tau$. If $\A$ is non-unital, then the Bures density space $(D_\tau(\A), d_B^\tau)$ is not compact.
\end{corollary}
\begin{proof}
    By \cite[Theorem III.1.1]{Davidson}, we have that $\A$ is infinite dimensional, and so the result follows from Theorem \ref{t:hb-bures}.
\end{proof}
Although this corollary implies that the density space in any non-unital C*-algebra equipped with a faithful trace must have a sequence with no converging subsequence with respect to the associated Bures metric, we still think its important to provide some examples of explicit sequences that have no converging subsequences in specific cases to provide more understanding of the Bures metric.

\begin{example}
    For each $n\in \N, x\in (0,1]$, let
\[ f_n(x) = \begin{cases} 
      2^{2(n+1)}(x-2^{-n}) & x\in (2^{-n},2^{-n}+2^{-(n+1)}) \\
      -2^{2(n+1)}(x-2^{-n})+2^{n+2} & x\in [2^{-n}+2^{-(n+1)},2^{1-n}) \\
      0 & \text{otherwise}
   \end{cases}
\]
(See Figure \ref{f:noncompact}).
\begin{figure}
    \centering
    \includegraphics[width=10cm]{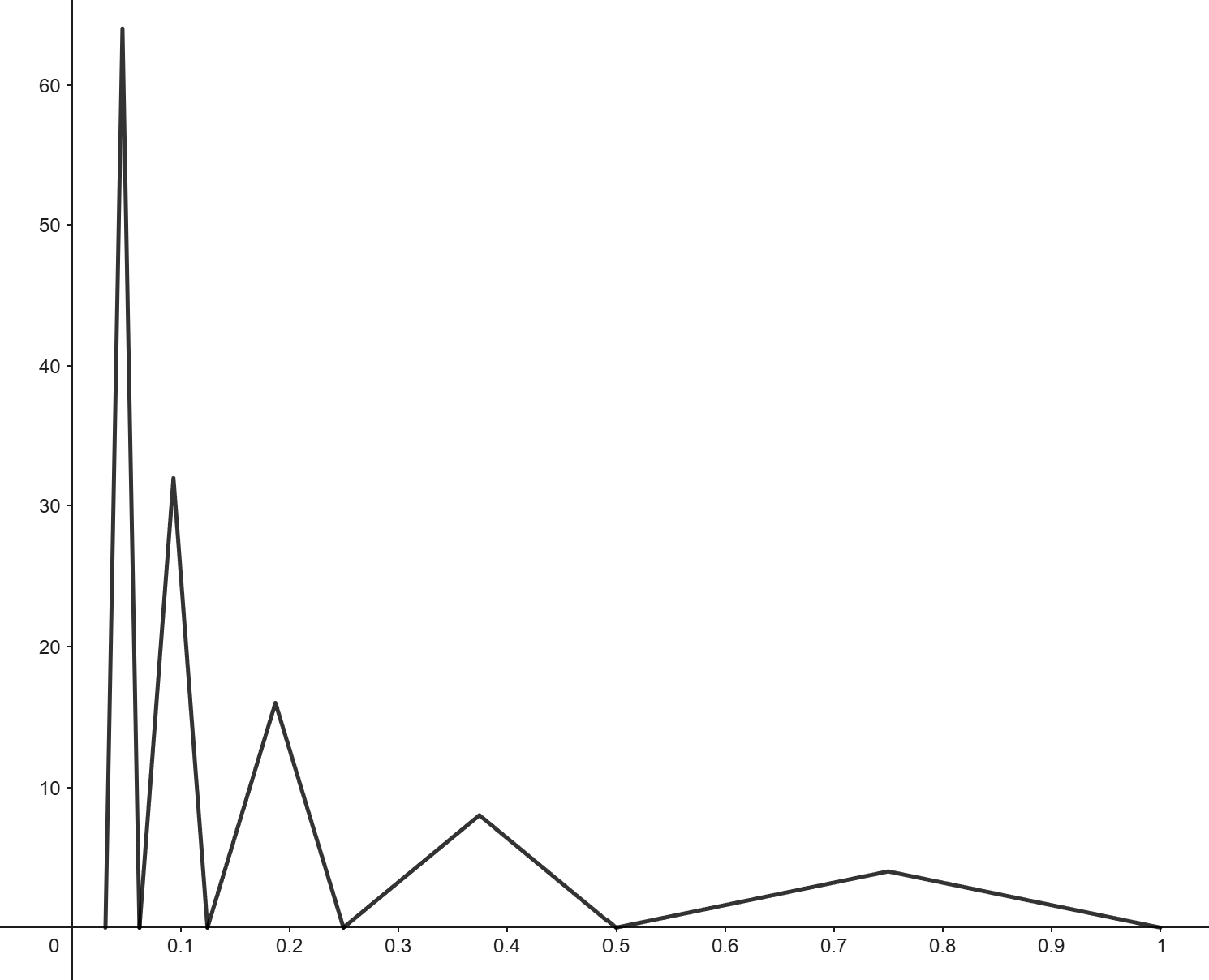}
    \caption{$f_1,f_2,f_3,f_4,f_5$}
    \label{f:noncompact}
\end{figure}
By construction, $f_n\in D_\la (C_0((0,1]))$ for every $n\in \N$ where $\la$ is the Lebesgue measure. Moreover, as each $f_n$ and $f_m$ have disjoint support for every $n,m\in \N, n\neq m$, then so do $\sqrt{f_n}$ and $\sqrt{f_m}$ for every $n,m\in \N, n\neq m$. Hence, for every $n,m\in \N, n\neq m$, we have
\[
F_\la(f_n,f_m)=\la(|\sqrt{f_n}\sqrt{f_n}|)=\la(0)=0,
\]
and thus
\[
d_B^\la(f_n,f_m)=\sqrt{1-0}=1.
\]
Therefore, $(f_n)_{n\in \N}$ has no converging subsequence with respect to the Bures metric. And so $(D_\la(C_0((0,1])), d_B^\la)$ is not compact (in fact, it is not even totally bounded since $(f_n)_{n\in \N}$ is a sequence with no Cauchy sequence). One can observe that this also provides an example for the unital (but still infinite dimensional) C*-algebra, $C([0,1])$
\end{example}

The examples so far have been of infinite-dimensional algebras that are not AF, and so, in the following example, we provide an AF example.

\begin{example}\label{e:quantized}
    In this example, we consider the non-unital C*-algebra $C_0(\N)$. For the faithful trace, we first define a standard faithful probability  measure on the power set $\sigma$-algebra on $\N$, which is just the Borel $\sigma$-algebra associated to the discrete topology. For each $A\subseteq \N$, define
    \[
    \om(A)=\begin{cases}
        \sum_{n\in A} \frac{1}{2^{n}} & A\neq \emptyset\\
        0 & A=\emptyset
    \end{cases}
    \]
    As an abuse of notation, let $\om$ denote the associate faithful trace on $C_0(\N)$ given by integration with respect to $\om$. For each $n\in \N$, let $\chi_n$ be the characteristic function of $\{n\}$ and define
    \[
    a_n=2^n \chi_n.
    \]
    It holds that $a_n\in D_\om (C_0(\N))$ for every $n\in \N$
    since $a_n(m)\geq 0$ for every $m\in \N$ and 
    \[
    \om(a_n)=\int_\N 2^n\chi_n \ d\om=2^n \int_\N \chi_{\{n\}} \ d \om=2^n \om(\{n\})=2^n\cdot \frac{1}{2^n}=1
    \]
    However, for every $n,m\in \N, n\neq m$, we have that $a_n$ and $a_m$ have disjoint support, and thus, similarly to the last example
    \[
    d_B^\om(a_n,a_m)=1.
    \]
    Hence, $(f_n)_{n\in \N}$ has no converging sequence with respect to $d_B^\om$.  And so, the metric space  $(D_\la(C_0(\N)), d_B^\om)$ is not compact (in fact, it is not even totally bounded since $(f_n)_{n\in \N}$ is a sequence with no Cauchy sequence).

    One can extend this example to the unital (but still infinite-dimensional) case of $C(\overline{\N})$, where $\overline{\N}=\N\cup \{\infty\}$ is the Alexandroff compactification of $\N$, and $\{\infty\}$ is given measure $0$. 
\end{example}
Now, for some noncommutative examples.
\begin{example}\label{e:quantized-N}
 One can    follow the same process Example \ref{e:noncomm-nonunital} using Proposition \ref{p:faithful-extend} applied to the previous two examples to provide many  explicit examples of sequences in Bures density spaces with no converging subsequences in unital or non-unital noncommutative spaces.
\end{example}

\section{The non-unital case of the quantum metric on the density space}\label{s:qm-non-u}

As shown in \cite{Aguilar25}, given a compact quantum metric space equipped with a faithful trace, one can induce a metric on the density space. However, this was only done for the unital case. Therefore, as done with the Bures metric in Section \ref{s:non-u-bures}, we provide a quantum metric on the density space in the non-unital setting. However, in full generality, the quantum metric on the density space in the non-unital setting will be an extended metric at best (meaning that it can take value $\infty$) as alluded to in Definition \ref{d:l-pair} but confirmed via Example \ref{e:extended}, but in this section, we  provide a natural sufficient condition for the quantum locally compact metric space in order for the quantum metric to be a metric (not   taking value infinity) by introducing the notion of a quantum Lipschitz triple, while also introducing many new families of quantum locally compact metric spaces in the both the commutative and noncommutative cases. We also confirm in Example \ref{e:non-qlt-comm} and Example \ref{e:non-qlt-noncomm} that the class of quantum Lipschitz triples is strictly contained in the class of quantum locally compact metric spaces.

We first provide further evidence that the definition of a quantum Lipschitz triple is satisfactory since the induced quantum metric topology is shown to be weaker than the topology induced by the C*-norm. Next, although by \cite[Proposition 3.3 and Corollary 3.6]{Aguilar25}  we have that the topology induced by the quantum metric on the density space is compact for finite-dimensional C*-algebras, it seems that the question of compactness in the infinite-dimensional setting regardless of whether the C*-algebra is unital or not is a lot more subtle than for the Bures metric where we were able to answer this question in full generality in the previous section. Similarly, it does not seem clear whether the Bures metric topology is topologically equivalent to the quantum metric topology in the infinite-dimensional setting even though they are topologically equivalent in the finite-dimensional setting by \cite[Corollary 3.6]{Aguilar25}. However, we finish this section by providing a non-unital (and thus infinite-dimensional) and some noncommutative families of examples for which the Bures metric topology and quantum metric topology on the density space are not uniformly equivalent and for which the quantum metric topology is not compact. We also use these examples to show that the C*-norm induced metric is also not uniformly equivalent to the quantum metric in general.

To define our quantum metric, we begin with the following proposition. The proof is very similar to that of \cite[Proposition 3.1]{Aguilar25} except for one key difference (the lack of a unit), and so, we include the whole proof.

\begin{proposition}\label{p:density-state-map}
Let $\A$ be a non-unital C*-algebra. Let $\tau$ be a faithful trace   on $\A$. 
The map 
\[
\Phi_\tau : a \in D_\tau(\A) \longmapsto \varphi_a^\tau \in S(\A)
\]
is well-defined and injective, where $\vhi^\tau_a(b)=\tau(ab)$ for every $a\in D_\tau(\A)$ and $b\in \A$. 
\end{proposition}
\begin{proof}
   Let $a\in \A_+$ and so there exists $b\in \A$ such that $a=b^*b$ by \cite[Lemma I.4.3]{Davidson}. Let $c\in \A$. Then since $\tau$ is a trace,
   \[
   \varphi_a^\tau(c^*c)= \tau(b^*bc^*c)= \tau(cb^*bc^*)= \tau(cb^*(cb^*)^*)\geq 0.
   \]
   Thus, $\varphi^\tau_a$ is a positive linear functional.

   Next, let $(e_\la)_{\la\in \La}$ be an approximate identity for $\A$, which exists by \cite[Theorem I.4.8]{Davidson}. By \cite[Lemma I.9.5]{Davidson}, we have that $\lim_\la \varphi_a^\tau(e_\la)=\|\varphi_a^\tau\|_\mathrm{op}$.  Hence as $a \in D_\tau(\A)$ and $\tau$ is continuous,
   \[
   \|\varphi_a^\tau\|_\mathrm{op}=\lim_\la \tau(ae_\la)=\tau(a)=1.
   \]
 Thus $\varphi_a^\tau$ is a state, and therefore, $\Phi_\tau$ is well-defined.

 Next, let $a,a'\in \A_+$ such that $\Phi_\tau(a)=\Phi_\tau(a')$. Hence $\varphi_a^\tau((a-a')^*)=\varphi^\tau_{a'}((a-a')^*)$. Thus $\tau(a(a-a')^*)=\tau(a'(a-a')^*)$, and so $\tau((a-a')(a-a')^*)=0$. Therefore, $a-a'=0$ as $\tau$ is faithful, which implies that $a=a'.$  Thus,  $\Phi_\tau$ is injective.
\end{proof}
This allows us to define a new metric on $D_\tau(\A)$, when $\A$ is non-unital, using quantum locally compact metric spaces (note that we only need a Lipschitz triple for this, and so,  the following still makes sense for a quantum locally compact metric spaces). Indeed

\begin{definition}\label{d:quantum-density-metric}
   Let $\A$ be a non-unital C*-algebra equipped with a faithful trace $\tau$. Let $(\A, L, \M)$ be a Lipschitz triple (Definition \ref{d:l-triple}). For every $x,y \in D_\tau(\A)$, define
    \[
    d_L^\tau(x,y)=mk_L(\Phi_\tau(x),\Phi_\tau(y))=mk_L(\varphi^\tau_x,\varphi^\tau_y),
    \]
    which defines an extended metric on $D_\tau(\A)$ since $\Phi_\tau$    is well-defined and injective by Proposition \ref{p:density-state-map}. We will still call $d_L^\tau$ a quantum metric.
\end{definition}

As mentioned above, it might be the case that the Monge-Kantorovich metric might not be finite on every pair of states (so, at best, an extended metric as seen in the following example). Now, this is not a flaw, but for our setting, it makes sense to desire a true metric (not just an extended metric) since we aim to compare our quantum metric on the density space with the Bures metric, which is not only a true metric but also a bounded metric (the condition we introduce in Definition \ref{d:q-l-triple} will in fact give us a {\em totally} bounded metric  as seen in Theorem \ref{t:qlt-metric}), and it would be not be a worthwhile task to compare an extended metric or even an unbounded metric with a bounded metric (the Bures metric). However, after looking through the literature, we did not find an explicit  example of a quantum locally compact metric for which its associated Monge-Kantorovich metric takes value infinity on a pair of states, and so we provide an example of this now since this is an issue of particular importance in this article.

\begin{example}\label{e:extended}
    Consider $(\R, \dd_1)$, where $\dd_1$ is the usual absolute value metric on $\R$. Since  $(\R, \dd_1)$ is a locally compact metric space, we have that the Lipschitz triple $(C_0(\R), L_{\dd_1} , C_0(\R))$ is a quantum locally compact metric space by \cite[Theorem 4.1]{Latremoliere12b}, where  $L_{\dd_1}$ is the Lipschitz constant associated to $(\R, \dd_1)$. That is
    \[
    L_{\dd_1}(f)=\sup \left\{\frac{|f(x)-f(y)|}{\dd_1(x,y)} : x,y \in \R,  x\neq y\right\}
    \]
    for every $f\in C_0(\R)$.

    Now, define
    \[
    h: x\in \R\longmapsto \frac{1}{\pi(1+x^2)}\in \R
    \]
    and consider the faithful measure $\la_h$ on $\R$ induced by $h$ and the Lebesgue measure $\la$. As $\int_\R h\  d \la =1$, we have that the associated faithful positive linear functional $\tau_{\la_h}$ given by integration with respect to $\la_h$ is a state. Also, consider the Dirac point mass at $0$, denoted $\de_0$ and defined by 
    \[
    \de_0(f)=f(0)
    \]
    for every $f\in C_0(\R)$. We will show that 
    \[
    mk_{L_{\dd_1}}(\tau_{\la_h}, \de_0)=\infty.
    \]

    For this, we define a sequence $(f_n)_{n\in \N}$ in $C_0(\R)$ such that $L_{\dd_1}(f_n)\leq 1$ and $\de_0(f_n)=0$ for every $n\in \N$. Let $n\in \N$ and let $x\in \R$, define
    \[
    f_n(x)=\begin{cases}
        0 & x\in (-\infty, 0)\\
       x  & x\in [0, \tan(\frac{\pi}{2}(1-2^{-n}))]\\
     g_n(x) & x\in ( \tan(\frac{\pi}{2}(1-2^{-n})), \infty)
    \end{cases}
    \]
    where $g_n$ is any non-negative continuous function on $( \tan(\frac{\pi}{2}(1-2^{-n})), \infty)$ such that $g_n(\tan(\frac{\pi}{2}(1-2^{-n})))=\tan(\frac{\pi}{2}(1-2^{-n}))$, $g_n$ is decreasing,  $L_{\dd_1}(g_n)\leq 1$, and $\lim_{x\to\infty} g_n(x)=0$ (one can just take a line with slope $-1$ that intersects $(\tan(\frac{\pi}{2}(1-2^{-n})),\tan(\frac{\pi}{2}(1-2^{-n}))$ and impose that it vanishes once it intercepts the $x$-axis).  Now
    \begin{align*}
       \tau_{\la_h}(f_n)&=\int_{[0,\tan(\frac{\pi}{2}(1-2^{-n}))]} x\cdot \frac{1}{\pi(1+x^2)} \ d\la(x) \\
        & \quad \quad + \int_{[\tan(\frac{\pi}{2}(1-2^{-n})), \infty)} g_n(x)  \cdot \frac{1}{\pi(1+x^2)} \ d\la(x) \\
        & \geq \int_{[0,\tan(\frac{\pi}{2}(1-2^{-n}))]} x\cdot \frac{1}{\pi(1+x^2)} \ d\la(x)+0\\
        & = \frac{1}{2\pi}\ln\left(\left( \tan\left(\frac{\pi}{2}(1-2^{-n})\right)\right)^2+1 \right),
    \end{align*}
    which tends to $\infty$ as $n \to \infty$ since $\frac{\pi}{2}(1-2^{-n}))$ tends to $\pi/2$ from the left. However, $\de_0(f_n)=0$. Thus for every $n\in \N$
    \begin{align*}
        mk_{L_{\dd_1}}(\tau_{\la_h}, \de_0)& = \sup \{|\tau_{\la_h}(f)-\de_0(f)| : f\in sa(C_0(\R))), L_{\dd_1}(f)\leq 1 \}\\
        & \geq |\tau_{\la_h}(f_n)-\de_0(f_n)|\\
        & = \tau_{\la_h}(f_n).
    \end{align*}
    And therefore, 
    \[
      mk_{L_{\dd_1}}(\tau_{\la_h}, \de_0)=\infty.
    \]
\end{example}
Hence, it is truly the case that one must be concerned if they are dealing with an extended metric or not. Thus, now that we have a good understanding of this finiteness issue, we find a natural  sufficient condition for when a Lipschitz triple provides a metric on the density space (not just an extended metric). But first, we show that this condition on the Lipschitz triple is enough for it to be a quantum locally compact metric space, so that we are   in the correct setting. For this, we define the following.

\begin{definition}\label{d:q-l-triple}
    Let $\A$ be a non-unital C*-algebra and let $(\A, L,\M)$ be a Lipschitz triple. We say that $(\A, L,\M)$ is a {\em quantum Lipschitz triple} if $(\wt{\A}, L)$ is a compact quantum metric space.
\end{definition}
Without looking at \cite{Latremoliere12b}, one might suspect that a quantum Lipschitz triple would be a good candidate for the definition of a quantum locally compact metric space. However, this definition would be too restrictive and not allow for the rich theory and examples developed in \cite{Latremoliere12b, Latremoliere25} (see Example \ref{e:non-qlt-comm} and Example \ref{e:non-qlt-noncomm} for   examples of some standard quantum locally compact metric spaces (a commutative and noncomumutative one)  that are not  quantum Lipschitz triples). Before we confirm that quantum Lipschitz triples are, in fact, quantum locally compact metric spaces, we provide some examples, which includes many commutative and noncommutative examples, and in particular, many homogeneous and approximately homogeneous C*-algebras. And combining the following result with Theorem \ref{t:q-l-triple-qlcms}, these examples provide many new families of quantum locally compact metric spaces.

\begin{proposition}\label{p:comm-ex}
    Let $(X, d)$ be a locally compact metric space such that $d$ extends to a metric $\wt{d}$ on the one-point compactification $\overline{X}=X\cup \{\infty\}$ that metrizes its topology. Let $\A=\overline{\cup_{n\in \N}\A_n}^{\|\cdot\|_\A}$ be a unital AF algebra equipped with faithful tracial state $\tau$ such that $\A_1=\C1_\A$, and for every $n\in \N$, $\A_n$ is a finite-dimensional C*-subalgebra of $\A$, $1_\A\in \A_n$, and $\A_n\subseteq \A_{n+1}$. Let $(\be(n))_{n\in \N}$ be a sequence of positive reals that converges to $0$.  Let $L_{\wt{d}}^{\A, \be}$ be the extended seminorm on $sa(C(\overline{X}, \A))$ of \cite[Theorem 3.10]{Aguilar19}.

    It holds that $(C_0(X, \A), L_{\wt{d}}^{\A, \be}, C_0(X, \C1_\A))$ is a quantum Lipschitz triple.
\end{proposition}
\begin{proof}
It is a routine exercise to prove that 
\[
(f,\al) \in \wt{C_0(X,\A)} \longmapsto f_\infty +\al 1_{C(\overline{X}, \A)}\in  C(\overline{X}, \A)
\]
where $f_\infty$ agrees with $f$ on $X$ and $f_\infty(\infty)=0$, is a *-isomorphism, and thus $C(\overline{X}, \A)$ is the standard minimal unitization of $C_0(X, \A)$, and $C_0(X, \A)$ can be identified with the following ideal of $C(\overline{X}, \A)$
\[
I_\infty = \{f \in C(\overline{X}, \A) : f(\infty)=0\}.
\] 
Let $(e_\la)_{\la\in \La}$ be an approximate identity for $C_0(X)$. For each $\la\in \La$ and $x\in X$, set 
\[
e^\A_\la(x)=e_\la(x)1_\A. 
\]
It holds that $(e_\la^\A)_{\la\in \La}$ is an approximate identity in $C_0(X, \C1_\A)$ for $C(\overline{X}, \A)$.

Finally, note that  $(C(\overline{X}, \A),  L_{\wt{d}}^{\A, \be})$ is a compact quantum metric space  by \cite[Theorem 3.10]{Aguilar19}, which completes the proof.
\end{proof}
Some examples that satisfy this hypothesis include any half-open interval $(a,b]$ equipped with the absolute value metric since its one-point compactification is just $[a,b]$. Also, one can metrize the one-point compactification  $\overline{\N}=\N\cup \{\infty\}$ of $\N$ with the metric defined for each $n,m\in \overline{\N}$
\[
d(n,m)= 
    \left|\frac{1}{2^{n-1}}-\frac{1}{2^{m-1}}\right|
\]
with the convention that $1/2^{m-1}=0$ if $m=\infty$, and $d$ restricted to $\N$ induces the discrete topology. For an example that doesn't satisfies this hypothesis, consider $(\R, \dd_1)$ where $\dd_1$ is the usual absolute value metric on $\R$. Now, the one-point compactification is metrizable (since it is homeomorphic to the circle), but any metric that extends $\dd_1$ on $\R$ to $\R\cup \{\infty\}$ could not metrize the Alexandroff topology since it is unbounded on $\R$, and of course, any metric that metrizes a compact topology cannot be unbounded (one can generalize this example to any unbounded locally compact metric space for which the one-point compactification is metrizable).  Furthermore, this provides an example that not only does not satisfy the hypothesis of Proposition \ref{p:comm-ex}, but it is not a quantum Lipschitz triple that's still a quantum locally compact metric space (see Example \ref{e:non-qlt-comm}).

Now, we confirm.
\begin{theorem}\label{t:q-l-triple-qlcms}
      Let $\A$ be a non-unital C*-algebra. If  $(\A, L,\M)$ is a quantum Lipschitz triple, then $(\A, L,\M)$ is a quantum locally compact metric space (\cite[Definition 3.1]{Latremoliere12b}).
\end{theorem}
\begin{proof}
    Let $\mu\in S(\wt{A})$.  By \cite[Proposition 1.3]{Ozawa05},  we have that 
    $B=\{x\in sa(\wt{\A}): L(x)\leq 1, \mu(x)=0\}$ is totally bounded since $(\wt{\A}, L)$ is a compact quantum metric space. Next, let $a,b\in \M$. As multiplication is uniformly continuous in each coordinate, we have that $aBb$ is totally bounded. Hence $(\A, L,\M)$ is a quantum locally compact metric space by \cite[Theorem 3.10]{Latremoliere12b}.
\end{proof}
 It is probably no surprise that the converse of the previous theorem is not true, but we still feel it is important to provide a couple examples to make this clear. The first example is a standard commutative example and the second is a standard noncommutative example.
\begin{example}\label{e:non-qlt-comm}
   For the quantum locally compact metric space of Example \ref{e:extended}, we have that  the  associated Monge-Kantorovich metric is not even finite between every pair of states, but every Monge-Kantorovich metric associated to a quantum Lipschitz triple is not just finite on every pair of states but, in fact, the metric is totally bounded (as we will see in Theorem \ref{t:qlt-metric}). We also note that the C*-algebra of this example is even separable, and so, separability is not even enough to provide a converse to the previous theorem, which is of note since the one-point compactification of $\R$ (which is homeomorphic to the circle) is metrizable, and so we see that the condition of Proposition \ref{p:comm-ex}, where we insist that the given metric extends to a metric on the one-point compactification of $X$ and not just that the one-point compactification is metrizable, is of value.
   \end{example}

\begin{example}\label{e:non-qlt-noncomm}
    For a noncommutative example of a quantum locally compact metric space that is not a quantum Lipschitz triple (that's also separable like the last example), the example $(K(\ell^2), L_\al, \mathcal{D})$ of \cite[Proposition 4.7]{Latremoliere12b}, where $K(\ell^2)$ is the compact operators on the infinite-dimensional separable Hilbert space $\ell^2$  and $\mathcal{D}$ is the diagonal of the compact operators, provides such an example. We first list some details of this example in order to prove that it is not a quantum Lipschitz triple. For each $n\in \N$, define
    \[
    P_n((x_m)_{m\in \N})=(x_1, x_2, \ldots, x_n, 0, 0,\ldots)
    \]
    for every $(x_m)_{m\in \N}\in \ell^2$. Let $(\al_n)_{n\in \N}$ be a sequence in $\R$ bounded below by some $\ga>0$. For each $a\in K(\ell^2)$, set
    \[
    L_\al (a)=\sup \{\al_n^{-1}\|P_naP_n\|_\mathrm{op} :n\in \N\},
    \]
    where $\|\cdot\|_\mathrm{op}$ is the operator norm. 
    Let $a\in K(\ell^2)$. Let $n\in \N$, then as projections are norm $1$, we have 
    \[
    \al_n^{-1}\|P_naP_n\|_\mathrm{op}\leq \ga^{-1}\|P_n\|_\mathrm{op}\cdot \|a\|_\mathrm{op}\cdot \|P_n\|_\mathrm{op}=\ga^{-1}\|a\|_\mathrm{op}.
    \]
    Hence
    \[
    L_\al(a)\leq \ga^{-1}\|a\|_\mathrm{op}.
    \]
    Thus $\dom{L_\al}=sa(K(\ell^2))$. Therefore, on the unitization $\wt{K(\ell^2)}$, where $L_\al$ is defined by $L_\al((a,\la))=L_\al(a)$ for every $(a,\la)\in sa(\wt{K(\ell^2)})$, we also have $\dom{L_\al}=sa(\wt{K(\ell^2)})$. However, if $(\wt{K(\ell^2)}, L_\al)$ were a compact quantum metric space, then since $K(\ell^2)$ is infinite dimensional, we must have that $\dom{L_\al}\neq sa(\wt{K(\ell^2)})$ by \cite[Corollary 2.2]{Aguilar-vBH-L}, a contradiction. Thus, $(K(\ell^2), L_\al, \mathcal{D})$ is an example of a quantum locally compact metric space that is not a quantum Lipschitz triple. We also note that since $K(\ell^2)$ has no faithful trace, it doesn't fit into the setting of the rest of the paper since the definition of the density space (Definition \ref{d:density-Bures}) requires a faithful trace, but we still think it was important to provide this example since we are introducing this new concept of quantum Lipschitz triples, which they themselves do not require a faithful trace to be defined.
\end{example}

Now that we have seen that quantum Lipschitz triples are quantum locally compact metric spaces, let's confirm that this condition on Lipschitz triples ensures that our quantum metric on the density space is indeed a metric, and in fact, a  bounded metric, which is desirable for future possible  comparisons with the Bures metric.
\begin{theorem}\label{t:qlt-metric}
     Let $\A$ be a non-unital C*-algebra equipped with a faithful trace $\tau$. If  $(\A, L,\M)$ is a quantum Lipschitz triple, then $d_L^\tau$ of Definition \ref{d:quantum-density-metric} is a metric on $D_\tau(\A)$. That is, $d_L^\tau(x,y)<\infty$ for every $x,y\in D_\tau(\A)$. Moreover, $(D_\tau(\A), d_L^\tau)$ is totally bounded and thus bounded.
\end{theorem}
\begin{proof}  Let $x,y\in D_\tau(\A)$. By \cite[Remark 2.9]{Latremoliere12b}, we have that 
    \begin{align*}
    d_L^\tau(x,y)&=mk_L(\varphi^\tau_x, \varphi^\tau_y)\\
    & = \sup \{ |\varphi^\tau_x(c)- \varphi^\tau_y(c)| : c\in sa(\A), L(c)\leq 1\}\\
    &= \sup \{ |\wt{\varphi^\tau_x}(c)- \wt{\varphi^\tau_y}(c)| : c\in sa(\wt{\A}), L(c)\leq 1\}.
\end{align*}
But as $\wt{\varphi^\tau_x}, \wt{\varphi^\tau_y}\in S(\wt{\A})$ by \cite[Proposition 2.1.5(vi)]{Dixmier-c-book}, we have that 
    \begin{align*}
    d_L^\tau(x,y)&=\sup \{ |\wt{\varphi^\tau_x}(c)- \wt{\varphi^\tau_y}(c)| : c\in sa(\wt{\A}), L(c)\leq 1\} 
     =mk_L(\wt{\varphi^\tau_x}, \wt{\varphi^\tau_y})<\infty
    \end{align*}
    since $(\wt{\A}, L)$ is a compact quantum metric space. Moreover, as $(S(\wt{\A}), mk_L)$ is a compact metric space, it is totally bounded, and as any subset of a totally bounded metric space is totally bounded, the proof is complete.
\end{proof}

Before we move on to our example that provides many insights, we provide more evidence that the notion of a quantum Lipschitz triple is a natural definition. Indeed, we will now show that for quantum Lipschitz triples the topology induced by the associated quantum metric is weaker than the topology induced by the C*-norm. This is a natural result since in the unital case, the quantum metric induces the weak* topology on the states, and so, one expects that any topology induced by the quantum metric should be somewhat weak, and in particular, weaker than the C*-norm topology.

\begin{theorem}\label{t:c*-norm-stronger-qm}
    Let $\A$ be a non-unital C*-algebra equipped with a faithful trace $\tau$. If  $(\A, L,\M)$ is a quantum Lipschitz triple, then the topology induced by $d_L^\tau$ of Definition \ref{d:quantum-density-metric} on $D_\tau(\A)$ is weaker than the topology induced by the C*-norm.
\end{theorem}
\begin{proof}
    Consider the state $\mu \in S(\wt{\A})$ defined by $\mu((b,\la))=\la$ for every $(b,\la)\in \wt{\A}$. Since $(\wt{A}, L)$ is a compact quantum metric space, we have that
    \[
    B=\{a \in sa(\wt{\A}) : L(a)\leq 1, \mu(a)=0\}
    \]
    is totally bounded by \cite[Proposition 1.3]{Ozawa05}, and thus bounded by some $K>0$. 

    Let $x,y\in D_\tau(\A)$. Let $(b,\la)\in B$ where $b\in \A$. Note that as $\mu((b,\la))=0$, We have that $\la=0$. Hence   
    \begin{align*}
        \left|\wt{\vhi_x^\tau}\left((b,0)\right)-\wt{\vhi_y^\tau}\left((b,0)\right) \right|&  = \left| \vhi_x^\tau(b)  - \vhi_y^\tau(b) \right| \\
        & = |\tau(xb)-\tau(yb)|\\
        & = |\tau((x-y)b)|\\
        & \leq \|\tau\|_\mathrm{op}\cdot \|(x-y)b\|_\A\\
        & \leq \|\tau\|_\mathrm{op}\cdot \|x-y\|_\A\cdot \|b\|_\A\\
        & \leq K\|\tau\|_\mathrm{op} \cdot \|x-y\|_\A.
    \end{align*}
    Thus, by \cite[Proposition 2.7]{Latremoliere12b},
    \[
    d_L^\tau(x,y)\leq K\|\tau\|_\mathrm{op} \cdot \|x-y\|_\A,
    \]
    which completes the proof.
\end{proof}
\begin{remark}
    We note that this previous result is also true in the unital case, but this was not mentioned in \cite{Aguilar25}. However, the proof is similar and shorter than the above proof for the unital case.
\end{remark}

 \subsection{Matrix-valued  functions on the quantized interval}\label{s:quantized}

For the remainder of this section, we focus our attention on one family of quantum Lipschitz triples. The C*-algebras of this section are a subfamily of the C*-algebras of Proposition \ref{p:comm-ex} including many noncommutative spaces, but the quantum metric structure is different in that we do not utilize the  the usual Lipschitz constant. This choice of using a different construction of an $L$-seminorm is not only because it provides the desired results (a topology that is not compact on the density space and a metric that is not uniformly equivalent to the Bures metric or the C*-norm induced metric), but also because we hope that it provides an example that could lead to future results   as the $L$-seminorm is still a natural construction but one that is not directly constructed from classical structure like in the case of the Lipschitz constant.

Let $N\in \N$. The C*-algebras we study in this section are
\[
C_0(\N, M_N(\C))=\{f:\N \rightarrow M_N(\C) \mid \lim_{n\to \infty}f(n)=0\},
\]
where the limit is being taken in the operator norm  $\|\cdot\|_\mathrm{op}$ on matrices, 
and note that for $N=1$, we have $C_0(\N, M_1(\C))=C_0(\N)$. 
The reason we call this section, "matrix-valued functions on the quantized interval," is because the one-point point compactification of $\N$ with respect to the discrete topology on $\N$, denoted $\overline{\N}$, is homeomorphic to the quantized interval
\[
\left\{ \frac{1}{2^{n-1}}\in \R : n\in \N\right\}\cup \{0\},
\]
which is equipped with the subspace topology induced from the usual topology on $\R$,  
but we choose to represent $\overline{\N}$ as simply $\N\cup \{\infty\}$ to match with Example \ref{e:quantized} to avoid introducing more notation than necessary.  

For notational purposes, for any locally compact Hausdorff space $X$, for each $f\in C_0(X, M_N(\C))$ and $j,k\in \{1, 2, \ldots, N\}$, we let $f_{j,k}: X\rightarrow \C$ denote the function defined for every $x\in \N$ by 
\[
f_{j,k}(x)=f(x)_{j,k}
\]
where $f(x)_{j,k}$ denotes the $(j,k)$-entry of $f(x)\in M_N(\C)$. We note that $f_{j,k}\in C_0(X)$ and we will sometimes denote
\[
f=(f_{j,k})_{j,k\in \{1,2,\ldots, N\}}.
\]

The faithful trace we will be using to define the density space in $C_0(\N, M_N(\C))$ is  formed by $\omega$ from Example \ref{e:quantized} and the trace $\frac{1}{N}\mathrm{Tr}$ and Proposition \ref{p:faithful-extend}, which we denote by $\om^N$. In particular,  a standard measure theory exercise shows that 
\[
\omega(f)=\sum_{n=1}^\infty f(n)\frac{1}{2^n}
\]
for every $f\in C_0(\N)$, and thus, for every $f\in C_0(\N, M_N(\C)$, we have that 
\[
\om^N(f)=\sum_{n=1}^\infty \frac{1}{N}\mathrm{Tr}(f(n))\frac{1}{2^n}=\sum_{n=1}^\infty \left(\sum_{k=1}^N\frac{1}{N}f_{k,k}(n)\frac{1}{2^n}\right).
\] 
We note that for $N=1$, we have that $\om^N=\om$.

Towards defining the $L$-seminorm on the unitization, we note that the standard minimal unitization of $C_0(\N, M_N(\C))$ can be identified with 
\[
C(\overline{\N}, M_N(\C))=\{f:\overline{\N}\rightarrow M_N(\C) \mid \lim_{n\to \infty}f(n)=f(\infty)\}
\]
(see the proof of Proposition  \ref{p:comm-ex}), 
where $\overline{\N}=\N\cup \{\infty\}$ is the one-point compactification of $\N$ and the limit again is with respect to convergence in the operator norm on $M_N(\C)$, which is equivalent to coordinate-wise convergence. Note that  $C_0(\N, M_N(\C))$ can be identified by the ideal of $C(\overline{\N}, M_N(\C))$ defined by 
\[
I_\infty=\{f\in C(\overline{\N}, M_N(\C)) : f(\infty)=0\},
\]
and 
\[
C(\overline{\N}, M_N(\C))=I_\infty \oplus \C1_{C(\overline{\N}, M_N(\C))},
\]
where $1_{C(\overline{\N}, M_N(\C))}$ is the constant $I_N$ (identity matrix) function, since for each $f\in C(\overline{\N}, M_N(\C))$, we have that $f$ uniquely can be written as \[f=(f-f(\infty)1_{C(\overline{\N}, M_N(\C))})+f(\infty)1_{C(\overline{\N}, M_N(\C))},\] where $(f-f(\infty)1_{C(\overline{\N}, M_N(\C))})\in I_\infty$ and $f(\infty)1_{C(\overline{\N}, M_N(\C))}\in \C1_{C(\overline{\N}, M_N(\C))}$. 

To define the $L$-seminorm on $C(\overline{\N}, M_N(\C))$,  we will use \cite[Theorem 3.5]{Aguilar-Latremoliere15}.    First, for the faithful tracial state, as an abuse of notation, we define
\[
\omega^N(f)=\sum_{n=1}^\infty \frac{1}{N}\mathrm{Tr}(f(n))\frac{1}{2^n}=\sum_{n=1}^\infty \left(\sum_{k=1}^N\frac{1}{N}f_{k,k}(n)\frac{1}{2^n}\right)
\]
for every $f\in C(\overline{\N}, M_N(\C))$, which is the unique extension of the above faithful trace $\omega^N$ on $C_0(\N, M_N(\C))$ to its unitization, $C(\overline{\N}, M_N(\C))$. The faithful state $\omega^N$ induces an inner product on $C(\overline{\N}, M_N(\C))$ defined for each $f,g\in C(\overline{\N}, M_N(\C))$ by
\[
\langle f,g\rangle_{\omega^N} =\omega^N(g^*f)=\sum_{n=1}^\infty \frac{1}{N}\mathrm{Tr}(g(n)^*f(n))\frac{1}{2^n}.
\]
 
For each $n\in \N$, define
\[
C_n^N=\{f\in C(\overline{\N}, M_N(\C)) : \forall k \geq n, f(k)=f(n)\},
\]
and  
\[
C_0^N=\{f\in C(\overline{\N}, M_N(\C)) : \exists \al\in \C, \forall  n\in \overline{\N}, f(n)=\al I_N\}=\C1_{C(\overline{\N}, M_N(\C))}.
\]
For all $n\in \N_0=\N\cup \{0\}$, note that $C_n^N$ is a finite-dimensional unital C*-subalgebra  of $C(\overline{\N}, M_N(\C))$. When $N=1$, we denote $C^1_n=C_n$ for every $n\in \N_0$, and note that $C_0=C_1$, but we keep this redundancy  in the $C(\overline{\N})$ case for notational purposes.

Next,
\[
C(\overline{\N}, M_N(\C))=\overline{\bigcup_{n=0}^\infty C_n^N}^{\|\cdot\|_\infty}
\]
by a Stone-Weierstrass argument. 
Let $n\in \N_0$. Let
\[
E^N_n : C(\overline{\N}, M_N(\C)) \rightarrow C^N_n
\]
be the orthogonal projection of $ C(\overline{\N}, M_N(\C)) $ onto $C_n^N$ with respect to the inner product $\langle \cdot, \cdot \rangle_{\omega^N} $, which is also the unique $\om^N$ preserving conditional expectation onto $C_n^N$ by \cite[Theorem 3.5 and Expression (4.1)]{Aguilar-Latremoliere15}. We note that 
\[
E^N_0(f)=\omega^N(f)1_{C(\overline{\N}, M_N(\C))}
\]  
for every $f\in C(\overline{\N}, M_N(\C))$. When $N=1$, we denote $E^1_n=E_n$ for every $n\in \N_0$ and note that $E_0=E_1$ as $C_0=C_1$.

Finally, let $(\beta(n))_{n\in \N}$ be a decreasing sequence of positive reals that converge to $0$ with respect to the usual topology on $\R$ such that $\be(0)=\be(1)$ (this is to avoid any issues in the $N=1$ case since  $C_0=C_1$). For each $f\in C(\overline{\N})$, define
\[
L^N_\beta(f)=\sup \left\{\frac{\|f-E^N_n(f)\|_\infty}{\beta(n)}: n\in \N_0 \right\}.
\]
We first confirm that we have a quantum Lipschitz triple.
\begin{proposition}
It holds that $(C_0(\N, M_N(\C)), L^N_\beta, C_0(\N, \C I_N))$ is a quantum Lipschitz triple.    Therefore, $d_{L^N_\be}^{\om^N}$ of Definition \ref{d:quantum-density-metric} is a totally bounded metric on the density space $D_{\om^N}(C_0(\N, M_N(\C)))$. 
\end{proposition}
\begin{proof}
     First,   $(C(\overline{\N}, M_N(\C)), L_\be^N)$ is a compact quantum metric space by \cite[Theorem 3.5]{Aguilar-Latremoliere15}. Moreover, $C_0(\N, \C I_N) $ contains an approximate identity for $C_0(\N, M_N(\C))$ (see the argument in Proposition \ref{p:comm-ex}). The last statement in the proposition is provided  by Theorem \ref{t:qlt-metric}.
\end{proof}
As mentioned above, we would like to be able to extend some results from \cite{Aguilar-vBH-L} on $C(\overline{\N})$ to $C(\overline{\N}, M_N(\C))$. To do this, we will relate $L_\be$ and $L_\be^N$ in the following result.
\begin{proposition}\label{p:amplified-cond-exp}
    For each  $n\in \N_0$ and $f\in C(\overline{\N}, M_N(\C))$, it holds that
    \[
    E_n^N(f)=(E_n(f_{j,k}))_{j,k\in \{1,2,\ldots, N\}}.
    \]
    Moreover, for every $f\in C(\overline{\N}, M_N(\C)) $ and $j,k\in \{1,2,\ldots, N\}$, we have 
    \[
    L_\be(f_{j,k})\leq L_\be^N(f).
    \]
\end{proposition}
\begin{proof}
    Let $n\in \N_0$. For every $f\in C(\overline{\N}, M_N(\C)) $, define
    \[
    F^N_n(f)=(E_n(f_{j,k}))_{j,k\in \{1,2,\ldots, N\}}.
    \]
    As conditional expectations are completely positive by \cite[Tomiyama's Theorem]{Brown-Ozawa}, we have that the linear map $F^N_n$ is positive since $F_n^N$ is an amplification of $E_n$. Therefore, it achieves its norm on the identity and as $F_n^N$ is unital by construction, we have that $F^N_n$ is contractive. Furthermore, $F^N_n$ is a projection onto $C_n^N$ as  $E_n$ projects onto $C_n$. Therefore, $F_n^N$ is a conditional expectation onto $C_n^N$ by \cite[Tomiyama's Theorem]{Brown-Ozawa}. 

    Now, as $E_n^N$ is the unique $\om^N$ preserving conditional expectation onto $C_n^N$ by \cite[Theorem 3.5]{Aguilar-Latremoliere15}, all that remains to complete the proof that $F^N_n=E_n^N$ is to show that $F_n^N$ preserves $\om^N$. Thus let $f\in C(\overline{\N}, M_N(\C))$. We have since $E_n$ is $\om$ preserving
    \begin{align*}
        \om^N(F_n^N(f))& = \om^N((E_n(f_{j,k}))_{j,k\in \{1,2,\ldots, N\}})\\
        & = \sum_{m=1}^\infty\left( \sum_{k=1}^N \frac{1}{N}E_n(f_{k,k})(m)\frac{1}{2^m}\right)\\
        & = \frac{1}{N}\sum_{k=1}^N \sum_{m=1}^\infty E_n(f_{k,k})(m)\frac{1}{2^m}\\
        & = \frac{1}{N}\sum_{k=1}^N \om(E_n(f_{k,k}))\\
        & = \frac{1}{N}\sum_{k=1}^N \om(f_{k,k})\\
        & = \frac{1}{N}\sum_{k=1}^N \sum_{m=1}^\infty f_{k,k}(m)\frac{1}{2^m}\\
        & = \om^N(f).
    \end{align*}
      Thus $F^N_n=E_n^N$.

Next, let $f\in C(\overline{\N}, M_N(\C))$.  Let  $x\in \overline{\N}.$ Fix $j,k\in \{1,2,\ldots, N\}$. As the max norm on $M_N(\C)$ is dominated by the operator norm by the table on page 365 of \cite{Horn}, we have that
\begin{align*}
    \left|f_{j,k}(x)-E_n(f_{j,k})(x)\right| & \leq \left\|f(x)-(E_n(f_{j,k}))_{j,k\in \{1,2,\ldots, N\}}(x)\right\|_{M_N(\C)}\\
    &= \left\|f(x)-F_n^N(f)(x)\right\|_{M_N(\C)}\\
    & = \|f(x)-E_n^N(f)(x)\|_{M_N(\C)}\\
   &  \leq \|f-E_n^N(f)\|_\infty
\end{align*}
Thus $   \|f_{j,k} -E_n(f_{j,k})\|_\infty \leq    \|f-E_n^N(f)\|_\infty $, which implies that
\[
L_\be(f)\leq L_\be^N(f)
\]
as desired.
\end{proof}

Next, the sequence $(a_n)_{n\in \N}$ of Example \ref{e:quantized} will play a pivotal role in the remaining results of this article: the non-uniform equivalence with the Bures metric and C*-norm induced metric, and the non-compactness of the quantum metric topology. But first, we need to adapt this sequence to the matrix setting and  establish some results about this matrix version of this  sequence.

For every $n\in \N$ and $x\in \N$, set 
\[
a_n^N(x)=a_n(x)I_N=2^n\chi_n(x)I_N,
\]
and note that $a^1_n=a_n$. And we note that $(a_n^N)_{n\in \N}$ is a sequence in $D_{\om^N}(C_0(\N, M_N(\C)))$ by Example \ref{e:quantized-N}.

\begin{lemma}\label{l:dirac}
    For every $n\in \N$, it holds that 
    \[
    \vhi^{\om^N}_{a^N_n}(f)=\frac{1}{N}\mathrm{Tr}(f(n))
    \]
for every $f\in C_0(\N , M_N(\C))$, where  $\vhi^{\om^N}_{a^N_n}$ is from Definition \ref{d:quantum-density-metric}. In particular, when $N=1$, we have 
    \[
    \vhi^\om_{a_n}=\de_n,
    \]
    where   $\de_n$ is the Dirac point mass at $n$ (i.e. $\de_n(f)=f(n)$ for every $f\in C_0(\N)$).
\end{lemma}
\begin{proof}
   Let $n\in \N$.  Let $f\in C_0(\N, M_N(\C))$. It holds that 
   \begin{align*}
    \vhi^{\om^N}_{a_n^N}(f) & = \om^N(a_n^N f)\\
    & = \sum_{m=1}^\infty \frac{1}{N} \mathrm{Tr}(a_n^N(m)f(m))\frac{1}{2^n} \\
    & = \frac{1}{N}\mathrm{Tr}(2^n I_Nf(n))\frac{1}{2^n}\\
    & = \frac{1}{N}\mathrm{Tr}(f(n))  
\end{align*}
as desired.
\end{proof}

\begin{lemma}\label{l:qm-calc}
    For every $n,m\in \N$ such that $n \neq m$, it holds that 
    \[
 d_{L_{\be}^N}^{\om^N}(a_n^N, a_m^N) \leq 2\max\{\be(n), \be(m)\}
    \]
    where $d_{L_{\be}^N}^{\om^N}$ is from Definition \ref{d:quantum-density-metric}, and thus, the sequence  $(a_n^N)_{n\in \N}$ is Cauchy with respect to $  d_{L_{\be}^N}^{\om^N}$.
\end{lemma}
\begin{proof}
    Let $n,m \in \N$ such that $n\neq m$.  With respect to $C_0(\N)$, we have by Lemma \ref{l:dirac} that 
    \[
    d_{L_\be}^\om(a_n,a_m)=mk_{L_\be}(\de_n,\de_m).
    \]
  By \cite[Remark 2.9]{Latremoliere12b},
    \begin{align*}
    mk_{L_\be}(\de_n,\de_m)&= mk_{L_\be}(\wt{\de_n},\wt{\de_m})\\
    & =  \sup\{|\wt{\de_n}(f)-\wt{\de_m}(f)| : f\in sa(C(\overline{\N})), L_\be(f)\leq 1\}\\
    \end{align*} 
Next, let $f\in sa(C(\overline{\N}))$. Then $f=(f-f(\infty)\mathds{1})+f(\infty)\mathds{1}$, where $(f-f(\infty)\mathds{1})\in I_\infty$, 
 and
 \begin{align*}
 \wt{\de_n}(f)& =\de_n(f-f(\infty)1_{C(\overline{\N})})+f(\infty)\\
 & =f(n)-f(\infty)1_{C(\overline{\N})}(n)+f(\infty)\\
 & =f(n)-f(\infty)+f(\infty)=f(n)=\de_n(f),
 \end{align*}
 and similarly $\wt{\de_m}(f)=f(m)=\de_m(f)$.  
 Hence 
\begin{align*}
& \sup\{|\wt{\de_n}(f)-\wt{\de_m}(f)| : f\in sa(C(\overline{\N})), L_\be(f)\leq 1\}\\
& \quad \quad \quad =\sup\{|\de_n(f)-\de_m(f)| : f\in sa(C(\overline{\N})), L_\be(f)\leq 1\}
\end{align*}
Consequently,
    \begin{align*}    mk_{L_\be}(\de_n,\de_m)  =2\max \{\be(n), \be(m)\}
    \end{align*}
    by \cite[Lemma 3.2]{Aguilar-vBH-L}.

    Now, let $f\in C_0(\N, M_N(\C))$ such that $L_\be^N(f)\leq 1$. Thus, by Proposition \ref{p:amplified-cond-exp}, we have that $L_\be(f_{j,k})\leq 1$ for every $j,k\in \{1,2,\ldots, N\}$. By Lemma \ref{l:dirac}, we have   
    \begin{align*}
        \left| \vhi^{\om^N}_{a_n^N}(f)-\vhi^{\om^N}_{a_m^N}(f)\right|& =\left| \frac{1}{N}\mathrm{Tr}(f(n))-\frac{1}{N}\mathrm{Tr}(f(m))\right|\\
        & =\frac{1}{N}\left| \sum_{k=1}^N (f_{k,k}(n)- f_{k,k}(m))\right|\\
        & \leq \frac{1}{N} \sum_{k=1}^N \left| f_{k,k}(n)- f_{k,k}(m)\right|\\
        & = \frac{1}{N} \sum_{k=1}^N \left| \de_n(f_{k,k})- \de_m(f_{k,k})\right|\\
        & \leq  \frac{1}{N} \sum_{k=1}^N mk_{L_\be}(\de_n, \de_m)\\
        & =  mk_{L_\be}(\de_n, \de_m)\\
        & = 2\max\{\be(n), \be(m)\}.
    \end{align*}
    Hence 
    \[
   d_{L_{\be}^N}^{\om^N}(a_n^N, a_m^N)= mk_{L_\be^N}\left( \vhi^{\om^N}_{a_n^N},\vhi^{\om^N}_{a_m^N} \right)\leq 2\max\{\be(n), \be(m)\}.
    \]
    Thus, as $(\be(n))_{n\in \N}$ converges to $0$, the proof is complete.
\end{proof}
And now, one of our main results of this section.   
\begin{theorem}\label{t:non-uniform}
    The Bures metric $d_B^{\om^N}$ of Definition \ref{d:density-Bures} and the quantum metric $d_{L^N_\be}^{\om^N}$ of Definition \ref{d:quantum-density-metric} are not uniformly equivalent on $D_{\om^N}(C_0(\N, M_N(\C)))$.
\end{theorem}
\begin{proof}
    By Example \ref{e:quantized-N}, we have that $(a^N_n)_{n\in \N}$ is not Cauchy with respect to $d_B^{\om^N}$ as $d_B^{\om^N}(a^N_n,a^N_m)=d^\om_B(a_n,a_m)=1$ for every $n,m\in \N,n \neq m$. However, $(a^N_n)_{n\in \N}$ is   Cauchy with respect to $d_{L^N_\be}^{\om^N}$ by Lemma \ref{l:qm-calc}. 
    Hence the identity map from \[(D_{\om^N}(C_0(\N, M_N(\C))),d_{L^N_\be}^{\om^N}) \] to  \[(D_{\om^N}(C_0(\N, M_N(\C))), d_B^{\om^N})\] is not uniformly continuous.
\end{proof}
We also show that the quantum metric   is not uniformly equivalent to the C*-norm induced metric on the density space.

\begin{theorem}
    The metric induced by the C*-norm and the  quantum metric $d_{L^N_\be}^{\om^N}$ of Definition \ref{d:quantum-density-metric} are not uniformly equivalent on $D_{\om^N}(C_0(\N, M_N(\C)))$.
\end{theorem}
\begin{proof}
    We have that $(a^N_n)_{n\in \N}$ is   Cauchy with respect to $d_{L^N_\be}^{\om^N}$ by Lemma \ref{l:qm-calc}. However, the sequence $(a^N_n)_{n\in \N}$ is unbounded with respect to the C*-norm $\|\cdot\|_\infty$, and is therefore not Cauchy with respect to its induced metric.  Hence the identity map from \[(D_{\om^N}(C_0(\N, M_N(\C))),d_{L^N_\be}^{\om^N}) \] to  \[(D_{\om^N}(C_0(\N, M_N(\C))), d_\A),\] where $d_\A$ is the C*-norm induced metric, is not uniformly continuous.
\end{proof}

For our final result, which is that $(D_{\om^N}(C_0(\N, M_N(\C))),d_{L^N_\be}^{\om^N})$ is not compact, we do a little more work, and first show that convergence with respect to the quantum metric $d_{L^N_\be}^{\om^N}$ is stronger than pointwise convergence.

\begin{proposition}\label{p:qm-pointwise}
    Let $(f_n)_{n\in \N}$ be a sequence in $D_{\om^N}(C_0(\N, M_N(\C)))$ and let $f\in D_{\om^N}(C_0(\N, M_N(\C))))$. If $(f_n)_{n\in \N}$ converges to $f$ with respect to $d_{L^N_\be}^{\om^N}$ of Definition \ref{d:quantum-density-metric}, then $(f_n)_{n\in \N}$ converges pointwise to $f$.
\end{proposition}
\begin{proof}
  Let $M\in \N$.  Fix a non-zero Hermitian $H\in M_N(\C)$. Let $\ee>0$. For every $x\in \N$, define
    \[
    \chi_{M,H}^N(x)=\begin{cases}
        H & x=M\\
        0 & x\neq M.
    \end{cases}
    \]Note that $\chi_{M,H}^N\in sa(C_0(\N, M_N(\C)))$. Moreover, as $\chi_{M,H}^N \in C_k^N$ for all $k \geq M$, we have that $\|\chi_{M,H}^N-E_k^N(\chi_{M,H}^N)\|_\infty=\|\chi_{M,H}^N-\chi_{M,H}^N\|_\infty=0$ for every $k \geq M$ as $E_k^N$ is a projection onto $C_k^N$. Thus $L_\be^N(\chi_{M,H}^N)<\infty$ by definition. Furthermore, as $\chi_{M,H}^N$ is not a scalar multiple of $1_{C(\overline{\N}, M_N(\C))}$, we have that $L_\be^N(\chi_{M,H}^N)>0$.
     
        Choose $P\in \N$ such that 
    \[
    d_{L_\be^N}^{\om^N}(f_n,f)< \frac{\ee}{N2^ML_\be^N (\chi_{M,H}^N)}
    \]
    for every $n \geq P$. Let $n \geq P$. Now, set 
    \[
    g=\frac{1}{L_\be^N(\chi_{M,H}^N)}\chi_{M,H}^N\in sa(C_0(\N, M_N(\C)))
    \]
    as $H$ is Hermitian, note that 
    and $L_\be^N(g)=1$ by construction. Thus, by definition of  $ d_{L^N_\be}^{\om^N}$, we have
that 
\[
 |\vhi^{\om^N}_{f_n}(g)-\vhi^{\om^N}_{f }(g)| \leq d_{L_\be^N}^{\om^N}(f_n,f).
\]
Hence
\begin{align*}
  &  |\mathrm{Tr}(f_n(M)H)-\mathrm{Tr}(f(M)H)|\\
    &=\frac{N2^ML_\be^N (\chi_{M,H}^N)}{N2^ML^N_\be (\chi_{M,H}^N)}|\mathrm{Tr}(f_n(M)H)-\mathrm{Tr}(f(M)H)|\\
    & = \frac{N2^ML_\be^N (\chi_{M,H}^N)}{N2^ML^N_\be (\chi_{M,H}^N)}|\mathrm{Tr}(f_n(M)\chi_{M,H}^N(M))-\mathrm{Tr}(f(M)\chi_{M,H}^N(M))|\\
    & = N2^ML_\be^N (\chi_{M,H}^N)\left|\frac{1}{N}\mathrm{Tr}(f_n(M)g(M))\frac{1}{2^M}-\frac{1}{N}\mathrm{Tr}(f(M)g(M))\frac{1}{2^M}\right|\\
    & =  N2^ML_\be^N (\chi_{M,H}^N)\left| \sum_{m=1}^\infty \frac{1}{N}\mathrm{Tr}(f_n(m)g(m))\frac{1}{2^m}-\sum_{m=1}^\infty \frac{1}{N}\mathrm{Tr}(f(m)g(m))\frac{1}{2^m}\right|\\
    & =N2^ML_\be^N (\chi_{M,H}^N)\left| \om^N(f_ng)-\om^N(fg)\right|\\
    & = N2^ML_\be^N (\chi_{M,H}^N)|\vhi^{\om^N}_{f_n}(g)-\vhi^{\om^N}_{f }(g)|\\
    & \leq  N2^ML_\be^N (\chi_{M,H}^N)\cdot d_{L_\be^N}^{\om^N}(f_n,f)\\
    & < N2^ML_\be^N (\chi_{M,H}^N) \cdot \frac{\ee}{N2^ML_\be^N (\chi_{M,H}^N)}=\ee.
\end{align*}
   Therefore, $(\mathrm{Tr}(f_n(M)H))_{n\in \N}$ converges to $\mathrm{Tr}(f(M)H)$. 
   
   Thus, \[(\mathrm{Tr}(f_n(M)H))_{n\in \N} \text{ converges to } \mathrm{Tr}(f(M)H)\] for any   self-adjoint $H\in M_N(\C)$. Hence, 
by linearity of $\mathrm{Tr}$ and the fact that the Hermitian matrices span $M_N(\C)$, we have that $(\mathrm{Tr}(f_n(M)A))_{n\in \N}$ converges to $\mathrm{Tr}(f(M)A)$ for any $A\in M_N(\C)$. Let $j,k\in \{1,2,\ldots,N\}$.  Let $E_{k,j}^N$ be the element in $M_N(\C)$ with $1$ in the $(k,j)$-entry and $0$ elsewhere. Since $\mathrm{Tr}(BE_{k,j})=B_{j,k}$ for every $B\in M_N(\C)$, we have that $(f_n(M)_{j,k})_{n\in \N}$ converges to $f(M)_{j,k}$. Thus, since operator norm convergence is equivalent to coordinate-wise convergence, we have that $(f_n(M))_{n\in \N}$ converges to $f(M)$ as desired.
\end{proof}

We can now prove our final main result for this paper.
\begin{theorem}\label{t:non-compact}
    The metric space $(D_{\om^N}(C_0(\N, M_N(\C))), d_{L^N_\be}^{\om^N})$ is not compact.
\end{theorem}
\begin{proof}
    Assume by way of contradiction that \[(D_{\om^N}(C_0(\N, M_N(\C))), d_{L^N_\be}^{\om^N})\] is compact. Then the sequence $(a^N_n)_{n\in \N}$ has a converging subsequence converging to some $f\in D_{\om^N}(C_0(\N, M_N(\C)))$ with respect to $ d_{L^N_\be}^{\om^N}. $  In fact, since $(a^N_n)_{n\in \N}$ is Cauchy with respect to $d_{L^N_\be}^{\om^N}$ by Lemma \ref{l:qm-calc}, we have that $(a^N_n)_{n\in \N}$ converges to $f$ with respect to $d_{L^N_\be}^{\om^N}$. Hence \[(a^N_n)_{n\in \N} \text{ converges pointwise to } f\] by Proposition \ref{p:qm-pointwise}. However, $(a^N_n)_{n\in \N}$ converges pointwise to the constant $0$ function, which is necessarily $f$ since pointwise limits are unique for functions with Hausdorff codomain. On the other hand,  since $f\in D_{\om^N}(C_0(\N, M_N(\C)))$, we have  that   \[1=\om^N(f)=\om^N(0)=0,\] which is a contradiction. Thus, the metric space $(D_{\om^N}(C_0(\N, M_N(\C))), d_{L^N_\be}^{\om^N})$ is not compact. 
\end{proof}

\bibliographystyle{amsplain}
\bibliography{thesis}

 \vfill
\end{document}